\documentclass[10pt,a4paper]{amsart}
\usepackage{amsfonts}
\usepackage{amsthm}
\usepackage{amsmath}
\usepackage{amscd}
\usepackage[latin2]{inputenc}
\usepackage{t1enc}
\usepackage[mathscr]{eucal}
\usepackage{indentfirst}
\usepackage{graphicx}
\usepackage{graphics}
\usepackage{pict2e}
\usepackage{epic}
\usepackage[normalem]{ulem}
\numberwithin{equation}{section}
\usepackage{epstopdf}
\usepackage{verbatim}
\usepackage{amssymb}
\usepackage{tikz-cd}

\usepackage{hyperref}
\AtBeginDocument{}

\theoremstyle{plain}
\newtheorem{Th}{Theorem}[section]
\newtheorem{Lemma}[Th]{Lemma}
\newtheorem{Cor}[Th]{Corollary}
\newtheorem{Prop}[Th]{Proposition}

\theoremstyle{definition}

\newtheorem{Rem}[Th]{Remark}
\newtheorem{?}[Th]{Problem}

\def\al{\alpha}

\def\w{\wedge}
\def\R{\mathbb{R}}
\def\C{\mathbb{C}}
\def\Q{\mathbb{H}}

\def\Lm{\Lambda}

\def\om{\omega}
\def\Om{\Omega}
\def\vp{\varphi}

\def\dds{\frac{\partial}{\partial s}}

\def\ip{\raise1pt\hbox{\large$\lrcorner$}\>}

\begin{document}
	
\title{Spin(7) metrics from K\"ahler Geometry}

\author[U. Fowdar]{Udhav Fowdar}
	
\address{University College London \\ Department of Mathematics \\
Gower Street \\
WC1E  6BT\\
London \\
UK}
	
\email{udhav.fowdar.12@ucl.ac.uk}
		
	
	
\begin{abstract} 
We investigate the $\mathbb{T}^2$-quotient of a torsion free $Spin(7)$-structure on an $8$-manifold under the assumption that the quotient $6$-manifold is K\"ahler.
We show that there exists either a Hamiltonian $S^1$ or $\mathbb{T}^2$ action on the quotient preserving the complex structure. Performing a K\"ahler reduction in each case reduces the problem of finding $Spin(7)$ metrics to studying a system of PDEs on either a $4$- or $2$-manifold with trivial canonical bundle, which in the compact case corresponds to either $\mathbb{T}^4$, a $K3$ surface or an elliptic curve. 
By reversing this construction we give infinitely many new explicit examples of $Spin(7)$ holonomy metrics. In the simplest case, our result can be viewed as an extension of the Gibbons-Hawking ansatz.

\end{abstract}

\maketitle
\tableofcontents

\section{Introduction}

In this paper we study the K\"ahler reduction of torsion free $Spin(7)$-structures. More specifically we consider a non-compact eight-manifold $N^8$ endowed with a torsion free $Spin(7)$-structure, determined by a closed $4$-form $\Phi$, which is invariant under the free action of a given two-torus. In general the quotient six-manifold $P^6$ is only almost K\"ahler. Under the further assumption that the almost complex structure is integrable i.e. $P^6$ is K\"ahler, we discover that it inherits naturally either a $\mathbb{C}^\times$ or $(\mathbb{C}^\times)^2$-action. This allows us to perform a K\"ahler reduction, in the sense that this is both a symplectic and holomorphic quotient, to a complex surface $M^4$ or a complex curve $\Sigma^2$. Our main result is that one can reverse this construction i.e. starting from a K\"ahler manifold $M^4$ or $\Sigma^2$ with some additional data we can construct a $\mathbb{T}^2$-invariant $Spin(7)$ metric. Put differently, our construction allows us to reduce the $Spin(7)$ system of PDEs on $N^8$ ($d\Phi=0$) to a simpler system on a K\"ahler manifold. 
By solving these equations  in special cases we give many new explicit, though incomplete, examples of metrics with holonomy \textit{equal} to $Spin(7)$. The precise statements of our main results are given in Corollary \ref{spin7cor} and Theorem \ref{secondmaintheorem}.\\

\textbf{Motivation.} In \cite{Apostolov2003} Apostolov and Salamon considered the problem of taking the circle reduction of a $G_2$ manifold $L^7$ under the assumption that the quotient manifold $P^6$ is K\"ahler. They discovered a surprisingly rich underlying geometry which led to the construction many explicit $G_2$ holonomy metrics. Due to the intricate relation between $G_2$ and $Spin(7)$ geometry, a natural question to ask is whether a similar construction also holds for $Spin(7)$ manifolds. This investigation is precisely what led to the current paper. 

An important problem in differential geometry is the construction of Ricci flat metrics in higher dimensions. Due to the complex nature of the PDEs this is a very challenging task. In dimensions $7$ and $8$, however, an often simpler problem is to construct $G_2$ and $Spin(7)$ holonomy metrics (which are automatically Ricci flat) as the PDEs are of first order. This remains nonetheless a daunting problem as in general one still has to deal with a system of ($49$ or $56$) PDEs cf. \cite{Bryant1987}. The innovative idea of Apostolov and Salamon was that if one could reduce such a problem to a problem in K\"ahler geometry then one could considerably simplify these equations. This is essentially due to fact that in K\"ahler geometry one can often use the $dd^c$-lemma to reduce complicated system of PDEs to one involving only a function, e.g. the K\"ahler potential as in Yau's proof of the Calabi conjecture \cite{YauproofofCalabi}. In our present context we shall primarily appeal to the local $dd^c$-lemma to reduce the $Spin(7)$ equations to a \textit{single} PDE.

We split our analysis in two cases corresponding to the $\C^\times$ and $(\C^\times)^2$ action on $P^6$.

Our construction, in the $\C^\times$ action case recovers the results of \cite{Apostolov2003} in the special situation when $N^8$ is the product of a $G_2$ manifold $L^7$ and a circle. The key point of our construction relies on the fact that  the K\"ahler assumption on $P^6$ implies that $M^4$ is endowed with a holomorphic symplectic form $\om_2+i\om_3$. The 
$\mathbb{T}^2$ bundle $N^8\to P^6$ is then determined, up to finite covers, by the cohomology classes $[\om_2], [\om_3] \in H^2(M,\mathbb{Z})$, ignoring factors of $2\pi$, whereas in \cite{Apostolov2003} the $S^1$ bundle $L^7\to P^6$ was determined by just one of these classes.
In the simplest instance we can take $M^4$ to be a hyperK\"ahler manifold, then our construction can be viewed an as extension of the Gibbons-Hawking ansatz, which itself  gives a way of constructing $4$-dimensional hyperK\"ahler metrics, to $Spin(7)$ metrics, see Corollary \ref{GibbonsHawkingCor1} and \ref{GibbonsHawkingCor2}. Thus, this gives an elementary way of constructing local examples of $Spin(7)$ metrics starting from just a positive harmonic function on an open set of $\R^3$. Another interesting aspect of our construction is that it can also be viewed as a special case of the $\mathbb{T}^3$ reduction of $Spin(7)$ metrics via multi-moments as described by Madsen in \cite{MadsenT3multimoment}. This $\mathbb{T}^3$ action is obtained in our setting by considering, in addition to the original $\mathbb{T}^2$ action, the horizontal lift of the $S^1\subset \C^\times$. The multi-moment map turns out to be the actual symplectic moment map for the K\"ahler reduction from $P^6$ to $M^4$. This shows that the name multi-moment map is indeed befitting.

In the case of a $(\C^\times)^2$ action, we show that a similar theory holds. We show that the general problem of constructing a $Spin(7)$ metric can be reduced to choosing a positive harmonic function and solving a single PDE on a (real) $2$-dimensional surface $\Sigma^2$ with trivial first Chern class. From this we are able to construct explicit examples of $Spin(7)$ metrics starting from just an elliptic curve and the punctured complex plane. In contrast to the previous situation, the horizontal lift of the $\mathbb{T}^2 \subset (\C^\times)^2$ action on $P^6$ does not preserve the $Spin(7)$-structure. Thus, our examples correspond to torus bundles over torus bundles, which we aptly call `nilbundles'. In particular, our examples differ from those discovered by Madsen and Swann in the context of toric $Spin(7)$ manifolds \cite{MadsenSwann} which instead have $\mathbb{T}^4$ symmetry.

Currently the most effective method of constructing non-compact $Spin(7)$ holonomy metrics involves evolving cocalibrated $G_2$-structures on some homogeneous spaces via the Hitchin flow \cite[Theorem 7]{Hitchin01stableforms}. Recently a new construction was given by Foscolo, which involves constructing $Spin(7)$ metrics on `small' circle bundles over the anti-self-dual orbibundle of  self-dual Einstein $4$-orbifolds \cite{Foscolo2019}. This is an extension of the `adiabatic limit' construction in the $G_2$ case cf. \cite{Foscolo2017}. Our result provides a new way of constructing more examples. Albeit incomplete, we nonetheless expect that the explicit nature of our metrics will be useful as testing ground for more general theories and also in future gluing constructions, similar to the one carried out in \cite{hsvz} in the hyperK\"ahler setting. The first, and simplest, example which fits into our construction was discovered by Gibbons, L\"u, Pope and Stelle (GLPS) in \cite{GibbonsKahler} by taking a `Heisenberg limit' of the Bryant-Salamon $Spin(7)$ metric on the spinor bundle of $S^4$. A different description of their example was also given in \cite{Salamon2003}. It was detailed study of this example that led to the results in this paper.\\
\textbf{Outline of paper.}

In section \ref{generalreduction} we carry out the $\mathbb{T}^2$ reduction of a torsion free $Spin(7)$-structure and describe the intrinsic torsion of the induced $SU(3)$-structure on the quotient six-manifold $P^6$, which is generally only an almost K\"ahler manifold. In section \ref{kahlerreduction} we impose that the $SU(3)$-structure is K\"ahler i.e. that the almost complex structure is in fact integrable. Assuming that $P^6$ is not a Calabi-Yau $3$-fold, we show that it is naturally equipped with Hamiltonian vector fields $U$ and $W$ which also preserve the complex structure. These vector fields can either span a line or a $2$-plane in the tangent space of $P^6$ thereby generating either an $S^1$ or $\mathbb{T}^2$ action. We consider the two cases separately: the former case being the focus of sections \ref{section4} to  \ref{nonconstant} and the latter that of sections \ref{thirdreduction} to \ref{nonconstant2} . 

In section \ref{section4}, we carry out the K\"ahler reduction to a four-manifold $M^4$ endowed with a holomorphic symplectic form. We then explain how one can invert this procedure and reconstruct $(N^8,\Phi)$ in exactly two cases:
one corresponding to the case when one of the circle bundles is trivial (which corresponds to the aforementioned Aspotolov-Salamon construction) and the second one when both circle bundles are non-trivial, see Proposition \ref{g2thm} and Corollary \ref{spin7cor}. We also explain how in this setting the local problem of finding $Spin(7)$ metrics is reduced to solving a \textit{single} second order PDE (for a $1$-parameter family of K\"ahler potentials) on an open set of $M^4 \times \R^+$. We give a general existence result in the case when we have real analytic initial data on $M^4$. In section \ref{constant} we then proceed to describe the simplest examples that can arise from our construction starting from suitable hyperK\"ahler four-manifolds. In section \ref{example1} we describe how the examples of Gibbons et al. fit into our setup and in section \ref{example2} we give new explicit examples of a $Spin(7)$ metrics. 
In section \ref{hypersurfaces} we explain how the simplest examples may also be obtained from the Hitchin flow of cocalibrated $G_2$-structures on certain nilmanifolds. In section \ref{nonconstant} we show that one can perturb the (K\"ahler potential of the) examples of section \ref{constant} to construct more complicated ones. We illustrate this construction by giving an explicit example of a $Spin(7)$ metric by perturbing the GLPS cohomogeneity one example to one which is not.

In section \ref{thirdreduction} we address the situation when the commuting vector fields $U$ and $W$ are orthogonal. We carry out once again a K\"ahler reduction but now to a complex curve $\Sigma^2$. In this case we reduce the local problem of constructing a $Spin(7)$-metric to choosing a positive harmonic function on $\Sigma^2$ and solving  a \textit{single} PDE on an open set of $\Sigma^2 \times \R$. By inverting this construction we construct more examples of $Spin(7)$ metrics in sections \ref{constant2} and \ref{nonconstant2}. We follow closely the strategy in the $\C^\times $ case. \\

\addtocontents{toc}{\protect\setcounter{tocdepth}{-1}}
\noindent\textbf{Acknowledgements.}
The author is indebted to his PhD advisors Jason Lotay and Simon Salamon for their constant support and many helpful discussions that led to this paper. The author would also like to thank Andrew Dancer and Lorenzo Foscolo for useful comments and corrections. 
This work was supported by the Engineering and Physical Sciences Research Council [EP/L015234/1]. The EPSRC Centre for Doctoral Training in Geometry and Number Theory (The London School of Geometry and Number Theory), University College London. 
\addtocontents{toc}{\protect\setcounter{tocdepth}{1}}
\section{$\mathbb{T}^2$-reduction of torsion free $Spin(7)$-structures}	\label{generalreduction}

\subsection{The basics} A torsion free $Spin(7)$-structure on an eight-manifold $N^8$ is defined by a closed $4$-form $\Phi$ which is pointwise linearly equivalent to 
\begin{align*}
\Phi_0 =\ &dx_{0123}+dx_{0145}+dx_{0167}+dx_{0246}-dx_{0257}-dx_{0347}-dx_{0356}\\
&dx_{2345}+dx_{2367}+dx_{4567}-dx_{1247}-dx_{1256}-dx_{1346}+dx_{1357},
\end{align*}
where $(x_0,\dots ,x_7)$ denote the standard coordinates on $\R^8$ and $dx_{ijlk}:=dx_i\w dx_j\w dx_k \w dx_l$. The $4$-form $\Phi$ then defines a Ricci-flat Riemannian metric $g_\Phi$, called a $Spin(7)$ metric, and volume form $vol_\Phi$ on $N^8$. Similarly a $G_2$-structure on a seven-manifold $L^7$ is defined by a $3$-form $\vp$ which is pointwise linearly equivalent to 
\[\vp_0=\partial_{x_{0}}\lrcorner\ \Phi_0, \]
with coordinates $(x_1,\dots, x_7)$ on $\R^7$. The $3$-form $\vp$ determines a metric $g_\vp$ and volume form $vol_\vp$, and hence also a Hodge star operator $*_\vp$ . We say that $\vp$ defines a torsion free $G_2$-structure if it is both closed and coclosed. The induced metric $g_\vp$, called a $G_2$ metric, is then Ricci-flat. We refer the reader to the standard references for more details on exceptional holonomy manifolds cf. \cite{Bryant1987, Joycebook, Salamon1989}. The last notion we shall require is that of an $SU(3)$-structure on a six-manifold $P^6$. This is given by an almost complex structure $J$, a real non-degenerate $2$-form $\om$ of type $(1,1)$ and a complex $3$-form $\Om=\Om^++i\Om^-$ of type $(3,0)$, satisfying the normalisation condition
\begin{equation}\frac{2}{3}\ \om^3= \Om^+\w\Om^-,\label{normalisation}\end{equation}
where we use the shorthand notation $\om^3=\om\w\om\w\om.$ If the differential forms $(\om,\Om)$ are closed then the induced metric $g_{\om}(\cdot,\cdot):=\om(\cdot,J\cdot)$ is Ricci-flat and $P^6$ is a Calabi-Yau $3$-fold. We denote by $*_\om$ the induced Hodge star operator. The theory of $SU(3)$-structures is elaborated in \cite{Bedulli2007, ChiossiSalamonIntrinsicTorsion}. 

Throughout this paper we shall use the notation $\|\cdot \|_{\om}$, $\|\cdot \|_{\vp}$ and $\|\cdot \|_{\Phi}$ to mean the pointwise norm (of vector fields and differential forms) defined by the metrics $g_\om$, $g_\vp$ and $g_\Phi$ respectively. Since our construction will be primarily local in nature any reference to the holonomy of a metric in this paper should be understood as the \textit{restricted} holonomy group.
\subsection{The general setup} In this paper we consider the problem of taking the quotient of a torsion free $Spin(7)$-structure $(N^8,\Phi)$ under the free action of a $2$-torus generated by $2$ \textit{orthogonal} vector fields (see remark \ref{remark2} below). Since $(N^8,\Phi)$ is Ricci-flat, the hypothesis that the action is free and preserves $\Phi$ implies that if $N^8$ is compact then it is locally the Riemannian product of a flat $\mathbb{T}^2$ and a six-manifold with holonomy contained in $SU(3)$. Thus, we shall assume that $N^8$ is non-compact, although our calculations are always valid in a small neighbourhood where such an action is free.

Denoting by $X$ and $Y$ the pair of orthogonal commuting vector fields generating this torus action, our hypothesis is that 
\[ \mathcal{L}_X\Phi=\mathcal{L}_Y\Phi=0. \]
The quotient six-manifold $P^6:=N^8/\mathbb{T}^2$ then inherits an $SU(3)$-structure. From a linear algebra point of view this follows from the fact that
\begin{gather*}
\frac{Spin(7)}{G_2}=S^7\ \ \ \text{ and }\ \ \ \frac{G_2}{SU(3)}=S^6,
\end{gather*}
whereby the $2$-plane in the tangent space of $N^8$ invariant under the $SU(3)$ action is generated by the span $\langle X,Y\rangle$. Let us denote by $(\om, \Om=\Om^++i\Om^-)$ the real $(1,1)$-form and complex $(3,0)$-form defining the induced $SU(3)$-structure on $P^6$. Our first task is to express $\Phi$ in terms of $(\om,\Om)$ and some suitable additional data. To do so we first define a $G_2$-structure on the seven-manifold $L^7$, obtained by quotienting $N^8$ by the circle action generated by the vector field $X$, by $\vp:=X\ip \Phi$. Viewing $N^8$ as an $S^1$ bundle over $L^7$ we can express $\Phi$ as
\[\Phi=\eta \w \vp + s^{4/3}*_\vp\vp,\]
where $s:=\|X\|^{-1}_\Phi=g_{\Phi}(X,X)^{-1/2}$ and $\eta(\cdot):= s^2 g_\Phi(X,\cdot)$. 
As $X$ and $Y$ commute and are orthogonal, the horizontal vector field $Y$ can also be interpreted as a vector field on $L^7$ generating an $S^1$ action. Thus, viewing $L^7$ as an $S^1$ bundle over $P^6$ we can express $\vp$ as
\[\vp=\xi \w \om + H^{3/2}\Om^+ ,\]
where $H:=\|Y\|^{-1}_\vp=g_{\vp}(Y,Y)^{-1/2}$ and $\xi(\cdot):= H^2 g_\vp(Y,\cdot)$. Geometrically, $\eta$ and $\xi$ correspond to connection $1$-forms on each $S^1$ bundle, and $s$ and $H$ are Higgs fields. This now allows us to write the $Spin(7)$ form $\Phi$ in terms of $( \om, \Om)$ as 
 \[  \Phi= \eta \w (\xi \w \om + H^{3/2}\Om^+)+s^{4/3}(\frac{1}{2}H^2 \om \w \om - H^{1/2}\xi \w \Om^-) .\]
 The above construction can be neatly summarised by the diagram:
\[ (N^8,\Phi,g_\Phi) \xrightarrow{/S^1_X} (L^7,\vp,g_\vp) \xrightarrow{/ S^1_Y} (P^6, \om, g_\om, \Om) . \]
Note that $\om$ and $\Om^+$ can equivalently be expressed as
\[ \om=Y\lrcorner \ X \lrcorner\ \Phi\ \ \ \text{ and }\ \ \ \Om^+=H^{-3/2}(X \lrcorner \ \Phi - \xi \w \om).\]
\begin{Rem}
A priori the reader might find it unnatural that we are distinguishing the vector fields $X$ and $Y$, since rather than performing a direct $\mathbb{T}^2$ reduction we are instead performing two circle quotients in succession. The advantage of this procedure of going through the intermediate $G_2$ quotient is that it makes it easier to reconcile our construction with the K\"ahler reduction of $G_2$ metrics \cite{Apostolov2003}.
\end{Rem}
The positive functions $s$ and $H$ are $\mathbb{T}^2$-invariant and as such are pullbacks of functions on $P^6$, which by abuse of notation we also denote by $s$ and $H$. The associated metrics are then related by:
\begin{align*} g_\Phi &= s^{-2}\eta^2 + s^{2/3}g_\vp,\\
g_\vp&=H^{-2}\xi^2+Hg_\om.
 \end{align*}
As the commuting vector fields $X$ and $Y$ both preserve the closed form $\Phi$ we have, using Cartan's formula,
\[ d\om=d(Y \ip X \ip \Phi)=\mathcal{L}_Y(X \ip \Phi)-Y \ip\mathcal{L}_X \Phi=X \ip \mathcal{L}_Y \Phi=0.\]
Furthermore, the condition $d\Phi=0$ implies that
\begin{gather}
d\Om^+=-\frac{3}{2}d(\ln H) \w \Om^+ - H^{-3/2}d\xi \w \om,\label{1condition}\\
d\Om^-=-(\frac{4}{3}d^c(\ln s) +\frac{1}{2}d^c(\ln H) )\w \Om^+ - s^{-4/3}H^{-1/2}d\eta \w \om,\label{2condition}\\
H^{3/2}d\eta \w \Om^++\frac{1}{2}d(H^2 s^{4/3}) \w\om^2 - s^{4/3}H^{1/3} d\xi \w\Om^-=0,\label{3condition}
\end{gather}
where $d^c:=J\circ d$ and $J$ is the almost complex structure on $P^6$ determined by $\Om$. Here we follow the convention that $J$ acts on a $1$-form $\beta$ by $J\beta(\cdot)=\beta(J \cdot)$, which differs from the convention in \cite{Apostolov2003} by a minus sign.

Note in particular that since $X$ preserves $\Phi$ it also follows that $\vp$ is a closed $G_2$ $3$-form. Moreover, assuming $s$ is not constant, from \cite[Theorem 3.6]{UdhavFowdar} we also know that $\vp$ is also coclosed, hence torsion free, only if $g_\Phi$ has holonomy contained in $SU(4)$.

From equations (\ref{1condition}) and (\ref{2condition}) it follows that $d\eta$ and $d\xi$ have no $\om$-component and  thus, $d\eta, d\xi \in \Lm^2_6 \oplus \Lm^2_8$, where $\Lm^2_6:=[\![\Lm^{2,0}]\!]$ and $\Lm^2_8:=[\Lm^{1,1}_0]$ following the notation of \cite{Bedulli2007}. We may then write
\begin{gather*}
d\eta \w\om = \al_\eta \w \Om^++(d\eta)^2_8 \w \om,\\
d\xi \w \om =\al_\xi \w \Om^++(d\xi)^2_8 \w\om,
\end{gather*}
for $1$-forms $\al_\eta$ and $\al_\xi$ on $P^6$, and with $(d\eta)^2_8$ and $(d\xi)^2_8$ denoting the $\Lm^2_8$-components of $d\eta$ and $d \xi $ respectively. Condition (\ref{3condition}) can then equivalently be expressed as
\[J(\al_\eta)-s^{4/3}H^{-1}\al_\xi=\frac{1}{2}H^{-3/2}d(H^2 s^{4/3}). \]
From the theory of $SU(3)$-structures cf. \cite{Bedulli2007,ChiossiSalamonIntrinsicTorsion} we can decompose the system (\ref{1condition}), (\ref{2condition}) into irreducible $SU(3)$-modules and express the $1$-forms $\al_\xi$ and $\al_\eta$ only in terms of $s$ and $H$. The result of this calculation is summed up in the following lemma:
\begin{Lemma}\label{torsionlemma}
The condition $d\Phi=0$ is equivalent to $d\om=0$ and the system
\begin{align*}
d\Om^+&= d(\ln (H^{-1/2} s^{-1/3}))\w \Om^+-H^{-3/2}(d\xi)^2_8 \w \om,\\
d\Om^-&= d^c(\ln (H^{-1/2} s^{-1/3}))\w \Om^+-s^{-4/3}H^{-1/2}(d\eta)^2_8 \w \om,
\end{align*}
with 
\[J(\al_\eta)=H^{1/2}s^{1/3}ds\ \ \ \text{ and } \ \ \
\al_\xi=-H^{1/2}dH+\frac{1}{3}H^{3/2}s^{-1}ds.\]
\end{Lemma}
\noindent In the notation of \cite{Bedulli2007}, the non-zero terms of the above $SU(3)$ decomposition are given by:
\begin{align*}
 \pi_1&=d(\ln(H^{-1/2} s^{-1/3})),\\
 \pi_2&=H^{-3/2}(d\xi)^2_8,\\
 \sigma_2&=s^{-4/3}H^{-1/2}(d\eta)^2_8.
\end{align*}
These differential forms define the intrinsic torsion of the $SU(3)$-structure $(\om,\Om)$ i.e. they measure the failure of the holonomy group to reduce to (a subgroup of) $SU(3)$ cf. \cite{Bryant1987, Salamon1989}. Similarly to the Gray-Hervella decomposition \cite{GrayHervella} one can define different classes of $SU(3)$-structures by imposing the vanishing of various combinations of these forms. In particular, we have the following interesting classes:
\begin{enumerate}
\item  Calabi-Yau (CY) i.e. $\pi_1=0$ and $\pi_2=\sigma_2=0$
\item  K\"ahler i.e. $\pi_2=\sigma_2=0$
\item  Special Generalised Calabi-Yau i.e. $\pi_1=0$ and $\pi_2=0$
\end{enumerate}
In this paper we shall be primarily interested in the K\"ahler case, but before proceeding ahead we make the following important observation.
\begin{Prop}
If $s$ is constant then $(L^7,\vp)$ has holonomy contained in $G_2$ and $(N^8,\Phi)$ is (locally) the Riemannian product of $L^7$ and $S^1$. If furthermore, $H$ is also constant then $(P^6,\om,\Om)$ has holonomy contained in $SU(3)$ and $(N^8,\Phi)$ is (locally) the Riemannian product of $P^6$ and a flat $2$-torus. Hence $\xi$ and $\eta$ cannot both be Hermitian Yang-Mills connections if $(N^8,\Phi)$ has holonomy $Spin(7)$.
\end{Prop}
\begin{proof}
If $s$ is constant then $d\eta \in \Lm^2_8$.  By differentiating the relation 
\[ d\eta \w \Om^-=0   \] 
we get that $\|d\eta\|_\om=0$. It follows that $[d\eta]$ defines a trivial class in $H^2(L^7,\mathbb{Z})$ and this proves the first claim. If $H$ is also constant we can apply the same argument to $d\xi$. The last assertion now follows directly from Lemma \ref{torsionlemma}.
\end{proof}
\begin{Rem}\label{remark2}\ \\
\textbf{1.} When the hypothesis that the vector fields $X$ and $Y$ are orthogonal is dropped the expressions appearing  in this section become more complicated. Furthermore, by imposing the K\"ahler assumption on the resulting $SU(3)$-structure we have only been able so far to find metrics with holonomy \textit{strictly} contained in $Spin(7)$. So we have decided to omit this more general setting in the present work to avoid introducing unnecessary complications. \\
\textbf{2.} Our construction also includes the $\mathbb{T}^2$ quotient of hyperK\"ahler eight-manifolds and CY $4$-folds under the group inclusions: $Sp(2)\subset SU(4) \subset Spin(7)$. As differential forms these can be expressed as
\begin{align}
\Phi 
&=\frac{1}{2}(\om_I \w \om_I+\om_J \w \om_J-\om_K \w \om_K)\\
&=\frac{1}{2}(\hat{\om}\w \hat{\om}) + Re(\hat{\Om}),\label{SU4spin7}
\end{align}
where $(\om_I,\om_J,\om_K)$ defines the hyperK\"ahler triple and $(\hat{\om},\hat{\Om})$ denotes the symplectic and holomorphic $(4,0)$-form of the CY $4$-fold.
Note that even if $N^8$ is a hyperK\"ahler manifold it is not generally the case that the quotient $SU(3)$-structure is torsion free. For instance, in \cite{UdhavFowdar} we considered the $\mathbb{T}^2$-quotient, generated by right and left multiplication by an imaginary quaternion, of (an open set in) $\mathbb{R}^8\cong \Q^2$ with the flat $Spin(7)$-structure and found that the quotient $SU(3)$-structure has all of the torsion forms $\pi_1$, $\pi_2$ and $\sigma_2$ non-zero.
\end{Rem}

\section{The K\"ahler reduction}\label{kahlerreduction}

\subsection{The first reduction} We shall now impose that $J$ is an integrable almost complex structure so that $(P^6,\om,J)$ is a K\"ahler manifold. This implies that $\pi_2=\sigma_2=0$ i.e. $d\eta, d\xi \in \Lm^2_6=[\Lm^{2,0}]$. Thus, by Lemma \ref{torsionlemma} we have that
\begin{align}
d\Om^+&= d(\ln (H^{-1/2} s^{-1/3}))\w \Om^+\label{kahler1}\\
d\Om^-&= d^c(\ln (H^{-1/2} s^{-1/3}))\w \Om^+\label{kahler2}
\end{align}
with 
\[J(\al_\eta)=H^{1/2}s^{1/3}ds\ \ \ \text{ and } \ \ \
\al_\xi=-H^{1/2}dH+\frac{1}{3}H^{3/2}s^{-1}ds\]
satisfying 
\[ [*_\om(\al_\eta \w \Om^+)] , [*_\om(\al_\xi \w \Om^+)]  \in H^{2}(P^6,\mathbb{Z}).\]
Note that the latter requirement is the natural higher dimensional analogue of the integrality-harmonic condition that figures in the Gibbons-Hawking ansatz, see section \ref{GibbonsHawkingAnsatz} below for comparison.
\begin{Rem}\
Since $$d(H^{1/2}s^{1/3}\Om)=0$$
i.e. $H^{1/2}s^{1/3}\Om$ is a \textit{holomorphic} $(3,0)$-form, it follows that the Ricci form of $(P^6,\om, \Om)$ is given by
\begin{align*} \rho&=i\partial \bar{\partial}(\ln (Hs^{2/3}))\\
&= i\partial \bar{\partial}(\ln H)+\frac{2}{3}i\partial \bar{\partial}(\ln s).
\end{align*}
and the scalar curvature is
\[ {S}=-d^{*}d(\ln(Hs^{2/3})), \]
where $d^*$ denotes the codifferential on $P^6$, cf. \cite[Pg. 158]{KobayashiNomizu2}.
\end{Rem}
\begin{Prop}\label{AS}
The intrinsic torsion $\tau$ of the closed $G_2$-structure $\vp$, defined by $d*_\vp\vp=\tau\w \vp$, is given by
\begin{align*}
\tau&=*_\om(\frac{1}{3}H^{1/2}s^{-1}d^cs\w\Om^+)-\frac{2}{3}H^{-1}s^{-1}\xi \w d^cs \\
&=-\frac{1}{3}s^{-4/3}d\eta -\frac{2}{3}H^{-1}s^{-1}\xi \w d^c s.
\end{align*}
\end{Prop}
\noindent Thus, it follows that the Apostolov-Salamon construction \cite{Apostolov2003}, which considers the K\"ahler $S^1$ reduction of torsion free $G_2$-structures, corresponds to the case when the first circle bundle is just a trivial bundle i.e. $\al_\eta=0$, or equivalently $d\eta=0$ or $s$ is constant (which by rescaling we can assume is $1$). In our notation Proposition $1$ of  \cite{Apostolov2003} can be stated as follows:
\begin{Prop}\label{g2thm}
Given a K\"ahler 6-manifold $(P^6,\om,J)$ with an $SU(3)$-structure determined by the $(3,0)$-form $\Om=\Om^+ + i\Om^-$ and a positive function $H$ such that 
\[ d(H^{1/2} \Om^+ )=0 \]
and 
\begin{equation}\label{topo2}  [-*_\om(\frac{2}{3} d(H^{3/2}) \w \Om^+  )] \in H^{2}(P,\mathbb{Z})   ,\end{equation}
then 
\[ \vp:= \xi \w \om + H^{3/2} \Om^+ \]
defines a torsion free $G_2$-structure on the $S^1$-bundle determined by (\ref{topo2}), where $\xi$ is a connection $1$-form on the circle bundle satisfying 
\[d\xi = -*_\om(\frac{2}{3} d(H^{3/2}) \w \Om^+  ) .\]
Moreover, the Hamiltonian vector field corresponding to $-H$ also preserves $\Om$, and hence $J$, and thus one can perform a K\"ahler reduction to a four-manifold endowed with a holomorphic symplectic structure.
\end{Prop}
Since we shall give explicit examples corresponding to the case when $s=H^{3/4}$ in sections \ref{example1} and \ref{example2}, it is worth stating the corresponding proposition in this situation.
\begin{Prop}\label{spin7thm}
Given a K\"ahler 6-manifold $(P^6,\om,J)$ with an $SU(3)$-structure determined by the $(3,0)$-form $\Om=\Om^+ + i\Om^-$ and a positive function $H$ such that 
\[ d(H^{3/4} \Om^+ )=0 \]
and 
\begin{equation}\label{topo}  [-*_\om(\frac{1}{2} d(H^{3/2}) \w \Om^+  )],\ [-*_\om(\frac{1}{2} d^c(H^{3/2}) \w \Om^+  )] \in H^{2}(P,\mathbb{Z})   ,\end{equation}
then 
\[ \Phi:= \eta \w \xi \w \om + H^{3/2}\eta \w \Om^+ + \frac{1}{2} H^3\om^2 -H^{3/2}\xi \w \Om^- \]
defines a torsion free $Spin(7)$ structure on the $\mathbb{T}^2$-bundle determined by (\ref{topo}), where $\eta$ and $\xi$ are connection $1$-forms on the torus bundle satisfying
\[d\xi = -*_\om(\frac{1}{2} d(H^{3/2}) \w \Om^+  ) ,\]
\[d\eta = -*_\om(\frac{1}{2} d^c(H^{3/2}) \w \Om^+  ). \]
\end{Prop}
\begin{proof}
	The proof is immediate from (\ref{kahler1}) and (\ref{kahler2}).
\end{proof}
\subsection{A second reduction}\label{secondreduction} In order to perform a further reduction, we define, in hindsight, two Hamiltonian vector fields $U$ and $W$ by
\begin{align*} \om(U, \cdot ) =- d(Hs^{-1/3}) \end{align*}
and 
\[\om(W,\cdot)=ds.\]
Using these vector fields, the curvature $2$-forms of $\eta$ and $\xi$ can be equivalently expressed as
\begin{gather}
	d\xi=-U \lrcorner\ (H^{1/2}s^{1/3}\Om^+)\label{connection1},\\
	d\eta=-JW \lrcorner\ (H^{1/2}s^{1/3}\Om^+)\label{connection2}.
\end{gather}
Thus, by differentiating these equations and using  (\ref{kahler1}) and (\ref{kahler2}) it follows that $U$ and $W$, in addition to being Hamiltonian, also preserve the complex structure $J$. In other words, they define an infinitesimal symmetry of the (torsion free) $U(3)$-structure determined by $(P^6,\om,J,g_\om)$. 
\begin{Rem}
It is known from \cite[Sect. 2]{Hitchin2000} that the stabiliser of the real (or imaginary) part of $\Om$ is isomorphic to $SL(3,\C)$. Since
\[SU(3)\cong SL(3,\C) \cap Sp(6,\R), \]
it follows that the $SU(3)$-structure is completely determined by the pair $(\om,\Om^+)$. Moreover the group inclusion
\[SL(3,\C)\hookrightarrow GL(3,\C) \]
implies in particular that changing $\Om^+$ by a positive factor leaves the induced complex structure $J$ unchanged. 
\end{Rem}
It is not generally true that $U$ and $W$ preserve the whole $SU(3)$-structure. In fact, we have that
\[\mathcal{L}_U\Om^+=\mathcal{L}_W\Om^+=0 \text{\ \ \ if and only if   }\ \ \mathcal{L}_Us=0. \]
We shall henceforth assume that this is indeed the case. The idea is now to perform a K\"ahler reduction using the action generated by these vector fields. We always work on an open set where $U$ and $W$ don't vanish. In particular we investigate the following two situations:
\begin{enumerate}
	\item $s=s(H)$ i.e. $s$ is a function of $H$\label{condition1}
	\item \label{condition2}$s$ and $y:=Hs^{-1/3}$ are independent functions, and the vector fields $U$ and $W$ are orthogonal i.e.
	\[ g_{\om}(U,W)=\om(W,JU)=0. \]
\end{enumerate}
Let us explain the geometry of these hypotheses. The assumption that $s$ is invariant by $U$ implies that $W$ and $JU$ are orthogonal. The two possibilities are either that $W$ lies in the complex span of $U$ and hence, $W$ and $U$ are equal up to some non-vanishing function, or that $W$ has a non-trivial component orthogonal to the span $\langle U,JU\rangle$. So geometrically condition (\ref{condition1}) is equivalent to saying that the vector fields $U$ and $W$ define the same line field on $P^6$, whereas condition (\ref{condition2}) is equivalent to saying that the complex planes defined by $\langle U,JU\rangle$ and $\langle W,JW\rangle$ are in fact orthogonal to each other. 

We shall consider these two cases separately, though our general strategy will the same in both cases. We first focus on situation (\ref{condition1}) and defer the study of case (\ref{condition2}) to section \ref{thirdreduction}.

\section{Further reduction \textrm{I}}\label{section4}
\subsection{$S^1$ K\"ahler reduction}\label{S1Kahkerreduction}
Working under the assumption that $s=s(H)$ we can perform a K\"ahler reduction, with respect to the vector field $U$, to a four-manifold $M^4$. The reader will find the general theory of K\"ahler reduction elaborated in \cite[Sect. $3$C]{HKLR}. We shall describe this construction in our context in more detail.

First we introduce a connection $1$-form $\alpha$ on $P^6$ by
\[ \al(\cdot)=u\ g_\om (U,\cdot), \]
where $u:=\|U \|^{-2}_\om$, so that $\alpha(U)=1$. From the definition of $U$, we can express $\alpha$ and $\om $ as
\begin{gather*}
\alpha = u g(H) s^{-1/3}d^cH,\\
\om = \tilde{\om}_1(H) + s^{-1/3} g(H)\al \w dH, 
\end{gather*}
where $g(H):=-1+\frac{1}{3}H s^{-1}s'$, with 
$'$ denoting the derivative with respect to $H$, and $\tilde{\om}_1$ is identified with the symplectic form on the  Marsden-Weinstein quotient $M^4$ of $(P^6,\om)$, with moment map $-Hs^{-1/3}$. We also define a holomorphic $(2,0)$-form $\om_2+i\om_3$, invariant under the complexified $U(1)$ action generated by the vector field $U$ on $M^4$, by
\[ H^{1/2}s^{1/3}\Om= (\om_2 +i \om_3)\w (\alpha - i J\alpha). \]
Viewed as a GIT or holomorphic quotient, a compatible complex structure $J_1$ on $M^4$ is defined by $\om_2(\cdot,\cdot)=\om_3(J_1\cdot,\cdot)$, cf. \cite[Sect. 8]{Salamon1989}. We are assuming here that the quotient is carried out for regular values of the moment map or equivalently that this is the stable GIT quotient. 

The final step of our construction is to express the K\"ahler constraint on $(\om,\Om)$ solely in terms of $u$, $\al$, $\tilde{\om}_1$, $\om_2$ and $\om_3$.

Denoting by $d_M$ and $d_P$ the exterior differential on $M^4$ and $P^6$ respectively, and defining 
$d_M^c:=J_1\circ d_M$, we find that condition $d_P\om=0$ is equivalent to
\[\tilde{\om}_1'=-s^{-1/3}g(H)d_M\al,\]
and condition $d_P(H^{1/2}s^{1/3}\Om)=0$ becomes
\[\al'=-g(H)s^{-1/3}d^c_Mu\]
and $d_M\al\w \om_2=d_M\al\w \om_3=0$. The result of combining these equations is (\ref{second}) whereas the normalisation condition (\ref{normalisation}) leads to the algebraic constraint (\ref{first}).

The fact that this construction is reversible follows by noting that given initial data on $M^4$ satisfying the above constraint we can define $N^8$ as the product of the $\mathbb{T}^3$ bundle determined by the \textit{integral} cohomology classes $[d\al], [d\xi]$ and $[d\eta]$, ignoring factors of $2\pi$ as usual, and a suitable interval $I^+_H$ with $H\in I^+_H \subset \R^+$. The result of this construction can be summed up in the following theorem.

\begin{Th}\label{generalAS}
	Let $(M^4,J_1)$ be a complex two-fold endowed with a $1$-parameter family of K\"ahler forms $\tilde\om_{1}(H)$, a $1$-parameter family of positive functions $u(H)$ and a holomorphic $(2,0)$-form given by
	$\om_2 + i \om_3 $
	satisfying the two conditions:
	\begin{gather}
	\frac{1}{2}u (\om_2 + i \om_3) \w (\om_2 - i \om_3) = H s^{2/3}\ \tilde{\om}_1 \w \tilde{\om}_1,\label{first}\\
	d_M d^c_M u ={s^{2/3}g^{-2}}\tilde{\om}_1'' + \frac{1}{2}({s^{2/3}}{g^{-2}})' \tilde{\om}_1',\label{second}
	\end{gather}
	for $H \in I^+_H \subset \R^+$. Then
	\begin{align*}\vp=\xi \w (\tilde{\om}_1+ s^{-1/3}g\ \al \w dH ) +Hs^{-1/3} \om_2 \w \al  - u H s^{-2/3}g\ \om_3 \w dH 
	\end{align*}
	defines a closed $G_2$-structure on $L^7$, the $\mathbb{T}^2_{\al,\xi}$ bundle determined by the integral cohomology classes $[d\xi]$ and $[d\al]$ on $M^4 \times I^+_H$, where 
	\[d\xi = -\om_2, \]
	\[d\al={s^{-1/3}g }\ d^c_Mu\w dH -s^{1/3}g^{-1} \tilde{\om}_1'  .\]
	If we further assume that
	\begin{equation}
	[*_\om( H^{1/2}s^{1/3}d^cs \w \Om^+)] \in H^2(P^6,\mathbb{Z})\label{integralcondition}
	\end{equation}
	so that there is another connection $1$-form $\eta $ satisfying
	\[d\eta = -*_\om( H^{1/2}s^{1/3}d^c_{P}s \w \Om^+) ,\]
	then the $4$-form
	\[ \Phi=\eta \w \vp + s^{4/3} *_\vp\vp \]
	defines a torsion free $Spin(7)$-structure on $N^8$, the total space of the $S^1$ bundle on $(L^7,\vp)$ defined by $[d\eta] \in H^2(P^6,\mathbb{Z})$, and the induced metric is given by
	\begin{equation}g_\Phi=s^{-2}\eta^2 +(s^{2/3}H^{-2})\xi^2 + (s^{2/3}Hu^{-1})\alpha^2+ (g^2Hu) dH^2+(s^{2/3}H)g_{\tilde{\om}_1}.\end{equation}
\end{Th}
\begin{Rem}\label{RemarkonGenericness}
For \textit{generic} data on $M^4$ satisfying the hypothesis of the theorem, the holonomy of $\Phi$ does not reduce to a subgroup of  $G_2$. If the holonomy is contained in $G_2$ then there exists a non-trivial parallel vector field, which also commutes with $X,Y$ and $U$ as they preserve $\Phi$ cf. \cite[Theorem $4$]{Bryant1989}. Since the curvature forms of $\eta$ and $\alpha$ are non-trivial unless $s$ is constant or $d_Mu=0$ and $\tilde{\om}_1'=0$, this vector field does not lie in the span of $\langle X,Y,U \rangle$ in general. Assuming this is the case, it must therefore descend to an infinitesimal symmetry of the K\"ahler structure on $M^4$ and $u(H)$. Thus, if we further assume that the data $(M^4,\om_1(H), \om_2+i\om_3, u(H))$ has no infinitesimal symmetry then from Berger's classification of holonomy groups we know (at least locally) that the holonomy must be \textit{either} $Spin(7)$, $SU(4)$ or $Sp(2)$. Note however that this is only a sufficient but not necessary condition as the horizontal lift of an infinitesimal symmetry of the data on $M^4$ will not generally preserve $\Phi$. Eliminating $SU(4)$ and $Sp(2)$ holonomy amounts to showing that there does not exist any parallel $2$-form cf. \cite{Bryant1987}.
\end{Rem}
\begin{Prop}
	The Ricci form of $(M^4,\tilde{\om}_1,J_1,g_{\tilde{\om}_1})$ is given by 
	\[ \rho_M= \frac{1}{2}d_Md_M^c(\ln u). \]
\end{Prop}
\begin{proof}
	This follows immediately from the fact that $$\|\om_2+i\om_3\|_{\tilde{\om}_1}=c_0\cdot u^{-1/2}H^{1/2}s^{1/3},$$
	where $c_0$ is a positive constant, and that $H$ and $s$ are constants on $M^4$.
\end{proof}
\noindent Thus, the induced metric on $M^4$ is Ricci-flat if and only if $\ln u$ is a harmonic function on $M^4$, for each value of $H$. In particular, if $M^4$ is compact then this means that $u$ is only a function of $H$ i.e. it is constant on $M^4$.

To sum up, what we have shown so far is that if a $Spin(7)$ manifold admits a $\mathbb{T}^2$-invariant $4$-form $\Phi$ with $s=s(H)$ and that the resulting quotient six-manifold is K\"ahler then in fact there exists a third $S^1$ action preserving the $Spin(7)$-structure. To be more precise, the horizontal lift of the vector field $U$ to $N^8$, still denoted by $U$ by abuse of notation, also preserves $\Phi$ since
\[   \mathcal{L}_U\eta=\mathcal{L}_U\xi=0,    \]
and commutes with $X$ and $Y$. In fact, our construction fits into the more general framework investigated by Madsen in the context of multi-moment maps on $Spin(7)$-manifolds with $\mathbb{T}^3$ symmetry \cite{MadsenT3multimoment}. In our present situation the multi-moment map $\nu: N^8 \to \R$, defined by
\[d\nu=\Phi(X, Y, U, \cdot\ ),\]
corresponds to the Hamiltonian function $-Hs^{-1/3}$ and the four-manifold $M^4$ can be identified with the ``multi-moment $Spin(7)$ reduction''. Our perspective has however the advantage of inheriting a richer structure owing to the K\"ahler condition which we shall exploit in the next sections. 

Note that one can generally solve equations (\ref{first}) and (\ref{second}) for different choices of the function $s$ and thus construct many closed $G_2$-structures admitting K\"ahler reduction, cf. \cite{GavinBall2020, UdhavFowdar3} for such examples. However, it is condition (\ref{integralcondition}) that determines when we can lift such a $G_2$-structure to a torsion free $Spin(7)$-structure. This is precisely what we investigate next i.e. we shall solve equation (\ref{integralcondition}) and thus determine for  which function $s(H)$ we get a torsion free $Spin(7)$ structure.
\subsection{The $Spin(7)$ condition}
From equations (\ref{connection1}) and (\ref{connection2}) the curvature forms can be expressed as:
\begin{gather*}
	d\xi= -s^{1/3}H^{1/2}(U \lrcorner\ \Om^+),\\
	d\eta=\frac{H^{1/2}s^{1/3}s'}{s^{-1/3}-\frac{1}{3}H s^{-4/3}s'}(JU \lrcorner\ \Om^+).
\end{gather*}
We also recall that the holomorphic $(2,0)$-form defined by
\begin{align*}
	\om_2 + i \om_3 &= \frac{1}{2}(U-iJU)\lrcorner \ (H^{1/2}s^{1/3})(\Om^+ + i \Om^-)\\
	&= H^{1/2}s^{1/3}((U \lrcorner\ \Om^+)+i(-JU \lrcorner\ \Om^+))
\end{align*}
is closed, since $d(H^{1/2}s^{1/3}\Om)=0$, and by definition is invariant on the leaves of the foliation generated by holomorphic vector field $U-iJU$ and thus passes to the K\"ahler quotient $M^4$. It is now easy to see that the curvature forms are equivalently given by
\begin{gather*}
	d\xi=-\om_2 \text{\ \ \ \ and\ \ \ \ }
	d\eta=-\Big(\frac{s'}{s^{-1/3}-\frac{1}{3}Hs^{-4/3}s'}\Big)\om_3.
\end{gather*}
\textbf{A remark on integrality and anti-instantons.} Although $\om_2$ and $\om_3$ do not generally define integral classes in $H^2(M,\mathbb{R})$ this is nonetheless always true locally (say on a small ball). 
For the rest of this paper we shall always assume that the classes are indeed integral and the reader is welcome to interpret the results as always valid  in a suitable open set. In the situation when $M^4$ is compact then, from the Enriques-Kodaira classification of complex surfaces cf. \cite[Chap. $7$]{Barth}, $M^4$ is either a torus $\mathbb{T}^4$ or a $K3$ surface with
\[[H^{0,2}\oplus H^{2,0}]\cap H^2(M,\mathbb{Z})=\langle [\om_2],[\om_3] \rangle_{\mathbb{Z}}.\]
The latter condition implies that the $K3$ surface has maximal Picard rank cf. \cite[($2.15$)]{Aspinwall2005} and \cite[Pg. $325$]{Barth}.
In particular, our assumption implies that the connection forms $\xi$ and $\eta $ are abelian \textit{anti-instantons}.

Thus, condition (\ref{integralcondition}) now becomes equivalent to solving the non-linear ODE:
\[ s'=A\cdot (s^{-1/3}-\frac{1}{3}Hs^{-4/3}s' ), \]
for $A \in \mathbb{Z}$. The solution is implicitly given by 
\begin{gather}\label{ODE1} AH=s^{1/3}(s+c) , \end{gather}
where $c$ is a constant of integration. If $A=0$ then the positivity assumption on $s$ forces $c$ to be negative, and by rescaling $s$ we can assume $c=-1$. Thus, $s=1$ and Proposition \ref{AS} implies that Theorem \ref{generalAS} reduces to \cite[Theorem 1]{Apostolov2003}. In other words, setting $A=0$ truncates the $Spin(7)$ equations to the $G_2$ equations considered in \cite{Apostolov2003}. Henceforth we will assume that constants $A \neq 0$ and $c$ are chosen such that $s$ and $H$, related by (\ref{ODE1}), are both positive. In what follows it will also be more convenient to use $s$ as the independent variable, rather than $H$. 
\begin{Cor}\label{spin7cor}
Given a four-manifold $M^4$ with the data $(\tilde\om_{1},\om_2 , \om_3, J_1,u)$ as in Theorem \ref{generalAS} and satisfying the two conditions:
\begin{gather}
\frac{1}{2}u (\om_2 + i \om_3) \w (\om_2 - i \om_3) = {A^{-1}s(s+c)}\ \tilde{\om}_1 \w \tilde{\om}_1,\label{equ1}\\
d_M d^c_M u =A^2 \frac{\partial^2}{\partial s^2}(\tilde{\om}_1)\label{equ2}
\end{gather}
for $s \in I^+_s \subset \R^+.$ Then the $4$-form
\[ \Phi=\eta \w \vp + s^{4/3} *_\vp\vp \]
defines a torsion free $Spin(7)$-structure on $N^8$, the total space of the $\mathbb{T}^3_{\al,\xi,\eta}$ bundle on $M^4 \times I^+_s$ defined by 
\begin{gather*}
d\xi = -\om_2,\\
d\eta=-A\cdot\om_3, \\
d\al=-{A^{-1}} d^c_Mu\w ds +A\dds (\tilde{\om}_1), 
\end{gather*}
and the induced metric is given by
\[g_\Phi=s^{-2}\eta^2+\frac{A^2}{(s+c)^2}\xi^2+\frac{s(s+c)}{A\cdot u}\alpha^2+\frac{s(s+c)u}{A^3}ds^2+\frac{s(s+c)}{A}g_{\tilde{\om}_1} .\]
Here the calibrated $G_2$-structure $\vp$ on $L^7$ is given by
\begin{align*}\vp=\ \xi \w (\tilde{\om}_1+ \frac{1}{A} ds \w \al) + \Big(\frac{s+c}{A}\Big)\om_2 \w \al \ + \frac{u(s+c)}{A^2}\ \om_3 \w ds. 
\end{align*}
\end{Cor}
\begin{Rem}
There exist in fact a $\mathbb{T}^2$-invariant $SU(4)$-structure on $N^8$ given by $(\hat{\om}, \hat{\Om})$ where 
\begin{gather*}
	\hat{\om}:=s^{-2/3}H^{-1}\eta \w \xi +s^{2/3}H\om,\\
	\hat{\Om}:=(s^{-1}\eta+iA(s+c)^{-1}\xi)\w(H^{3/2}s\Om),
\end{gather*}
which induces the torsion free $Spin(7)$-structure $\Phi$ of Corollary \ref{spin7cor} by expression (\ref{SU4spin7}). One verifies directly that neither $\hat{\om}$ nor $\hat{\Om}$ are closed and hence this $SU(4)$-structure is not torsion free. Furthermore, computing $d\hat{\Om}$ shows that it has non-vanishing component in $\Lm^{3,2}$, essentially corresponding to the Nijenhuis tensor, and as such the induced almost complex structure on $N^8$ is non-integrable. By contrast, Theorem $1$ of \cite{Goldstein2004} asserts that $N^8$ admits another $\mathbb{T}^2$-invariant  $(4,0)$-form given instead by
\[\check{\Om}:=(\eta-iA\xi )\w (H^{1/2}s^{1/3}\Om)\]
which is \textit{holomorphic}, and moreover if either $\om_2$ or $A\cdot\om_3$ defines a non-trivial cohomology class in $H^2(P^6,\mathbb{Z})$, then $N^8$ admits no compatible K\"ahler metric. Hence if either $\om_2$ or $A\cdot \om_2$ is cohomologically non-trivial then $g_\Phi$ defined by Corollary \ref{spin7cor} cannot be induced from a torsion free $\mathbb{T}^2$-invariant $SU(4)$-structure.
\end{Rem}
Before proceeding to the construction of explicit examples of $Spin(7)$ metrics, we first give a general existence result. Since this will be local in nature we shall without loss of generality set $A=1$. 
\subsection{A general local existence result}
Since $\tilde{\om}_1$ is a $1$-parameter family of K\"ahler forms there exists, on each suitably small open set $B \subset M^4$,  a K\"ahler potential $f:B\to \R$, depending on $s$ in a small interval, such that
\[ \tilde{\om}_1=d_Md_M^cf.\]
Thus, we may always solve equation (\ref{equ2}) by setting
\[u=\ddot{f}+\ddot{r}(s),\]
where $\dot{\ }$ refers to derivative with respect to $s$ and $\ddot{r}$ is a non-negative function of $s$ \textit{only}, chosen to ensure that $u$ is positive. Picking complex Darboux coordinates $(z_1,z_2)$ on $B$, we can express the $(2,0)$-form as
\[\om_2+i\om_3=\ dz_1 \w dz_2.\]
Defining $F:=f+r$, equation (\ref{equ1}) can now be expressed in these coordinates as
\begin{equation}
\Big(\frac{1}{s(s+c)}\Big)\ddot{F}=4 \det\Big(\frac{\partial^2}{\partial z_i \partial \bar{z}_j}F\Big)_{1\leq i,j\leq 2}  ,\label{mongeampere}
\end{equation}
where $\frac{\partial}{\partial z_j}:=\frac{1}{2}(\frac{\partial}{\partial x_j}+i\frac{\partial}{\partial y_j})$ for $j=1,2$. Under the assumption that we are given real analytic initial data to (\ref{mongeampere}), we may then appeal to the Cauchy-Kovalevskaya theorem for the general existence and uniqueness of a real analytic solution. 
\begin{Cor}
 Given a real analytic K\"ahler potential $F_{0}$ on (an open set of) a complex surface $(M^4,J_{1},\om_2+i\om_3)$ and an additional real analytic function $F_1$, then there exists a unique real analytic solution $F(s)$, for $s$ in a small interval, to (\ref{mongeampere}) with $F(0)=F_{0}$ and $\dot{ F}(0)=F_1$, and hence by Corollary \ref{spin7cor} a  torsion free $Spin(7)$-structure.
\end{Cor}
\begin{Rem}
Thus, we have abstractly proven that there exists a large class of $Spin(7)$ metrics admitting K\"ahler reduction. Our general solution is determined by $2$ functions, namely the two initial conditions to (\ref{mongeampere}), of $4$ variables. By contrast, Bryant shows using Cartan-K\"ahler theory that an arbitrary $Spin(7)$ metric is determined by $12$ functions of $7$ variables cf. \cite{Bryant1987}. This difference is essentially due to the fact that the $\mathbb{T}^2$ symmetry together with K\"ahler condition has allowed us to reduce the general problem to a \textit{single} second order PDE involving a family of K\"ahler potentials. 
\end{Rem}
For the sake of concreteness, we shall now investigate special cases when the pair (\ref{equ1}) and (\ref{equ2}), or equivalently (\ref{mongeampere}), can be solved explicitly.
\section{Constant solutions $\mathrm{I}$}\label{constant}
In this section we consider the simplest case in Corollary \ref{spin7cor} when $u$ is only a function of $s$ i.e. $d_Mu=0$. We show that one can find a large class of examples of $Spin(7)$ metrics simply by choosing suitable constants.

Assuming $d_Mu=0$, we solve equation (\ref{equ2}) to get
\[ \tilde{\om}_1=s\check{\om}_0 + \hat{\om}_0 ,\]
where $\check{\om}_0$ and $\hat{\om}_0$ are closed $2$-forms on $M^4$, independent of $s$. It is well-known that the wedge product on $\Lm^2:=\Lm^2(TM)$ defines a non-degenerate symmetric bilinear form $B$ of signature $(3,3)$ given by
\begin{align*}
	S^2(\Lm^2)  &\to \Lm^4 \cong \R \\
	(\beta_1,\beta_2) &\mapsto B(\beta_1, \beta_2)\theta,
\end{align*}
where $\theta:=\frac{1}{2}\om_2 \w \om_2=\frac{1}{2}\om_3 \w \om_3$. The orientation form $\theta$ on $M^4$ allows for a splitting 
$$\Lm^2=\Lm^2_+ \oplus \Lm^2_-,$$ 
with $B$ positive definite on $\Lm^2_+$ and negative definite on $\Lm^2_-$. 
Restricting $B$ to the $2$-plane in $\Lm^2$ spanned by $\langle \check{\om}_0, \hat{\om}_0 \rangle$, it follows from the theory of four-manifolds, cf. \cite[Chap. 7]{Salamon1989}, together with the fact that $\tilde{\om}_1$ is of type $(1,1)$ and $\om_2+i\om_3$ of type $(2,0)$  that there exist $(1,1)$-forms $\om_0$ and $\om_1$ spanning $\langle \check{\om}_0, \hat{\om}_0 \rangle$, and functions $a,b,p,q$ on $M^4$ such that
\begin{gather}\label{symplectic}
	\tilde{\om}_1=(a+bs)\om_0+(p+qs)\om_1,
\end{gather}
where,
\begin{gather*}
	\frac{1}{2}\om_0 \w \om_0 = - \theta, \ \ \  \frac{1}{2}\om_1 \w \om_1 =  \theta, \ \ \ \om_0 \w \om_1 = 0.
\end{gather*}  
From equation (\ref{equ1}) we have that 
\begin{equation} u=\frac{s(s+c)}{A}\cdot ((p+qs)^2-(a+bs)^2) ,\label{uequation}\end{equation}
and the fact that $\tilde{\om}_1$ defines a positive definite metric implies that we need
\begin{equation} p+qs > |a+ bs|  .\label{inequality}\end{equation}
Furthermore our hypothesis that $d_Mu=0$ forces $a,b,p,q$ to be constants. Hence it follows that the triple $(\om_1,\om_2,\om_3)$ defines a hyperK\"ahler structure on $M^4$, while $\om_0$ is a closed anti-self-dual $2$-form.
The $\mathbb{T}^3_{\al,\xi,\eta}$ bundle on $M^4$ is then determined by
\begin{equation}\label{curvatureforms} 
 (d\alpha, d\xi,d\eta)=(A\cdot(b\om_0 + q \om_1),-\om_2,-A\cdot \om_3) . 
\end{equation}
A trichotomy of the total space of the $\mathbb{T}^3$ bundle on given $M^4$ arises from whether $b>q$, $b=q$ or $b<q$. 
We can summarise the above into:
\begin{Th}\label{constantthm}
Given a hyperK\"ahler four-manifold $(M^4, g_{M}, \om_1,\om_2,\om_3)$ with an almost K\"ahler form $\om_0$, compatible with $g_{M}$ but inducing the opposite orientation, and suppose that the constants $(a,b,p,q)$ and the function $u(s)$ satisfy (\ref{inequality}) and (\ref{uequation}) respectively, then
\begin{align*}
\Phi =\ &\eta \w \Big(\xi \w (\tilde{\om}_1+\frac{1}{A} ds \w \al) +\frac{(s+c)}{A}(\om_2 \w \al)+\frac{u\cdot (s+c)}{A^2}(\om_3 \w ds)\Big)\\
+\ &\frac{s^{2}(s+c)^2}{A^2}\Big( \frac{1}{2}\tilde{\om}^2_1+\frac{1}{A}(\tilde{\om}_1 \w ds\w \al)  \Big) -s\cdot \xi \w (\om_3 \w \al - \om_2 \w \frac{u}{A}ds )
\end{align*}
defines a torsion free $Spin(7)$-structure on the product of the $\mathbb{T}^3$ bundle determined by (\ref{curvatureforms}) and $I^+_s$. For the case when $a=b=0$ we can set $\om_0=0$.
\end{Th}
We shall refer to examples arising from Theorem \ref{constantthm} as constant solutions. To exhibit the user-friendliness of the above theorem we shall give several explicit examples in the next two sections. In particular we consider $M^4=\mathbb{T}^4$ with the flat hyperK\"ahler structure, in which case it is well-known that the $\mathbb{T}^3_{\al,\xi,\eta}$ bundles correspond to certain $2$-step nilmanifolds cf. \cite{PalaisTorusbundle}. 
\section{Examples with holonomy $Spin(7)$, $G_2$, $SU(3)$ and $SU(4)$.}\label{example1}
Our aim here is to describe certain special holonomy metrics admitting K\"ahler reductions. In particular, the $Spin(7)$ metrics described in this section arise from Theorem \ref{constantthm} with $a=b=0$. We shall also explain how all these metrics can be obtained via a generalised version of the Calabi ansatz.
\subsection{The GLPS examples} 
In \cite{GibbonsKahler} Gibbons, L\"u, Pope and Stelle (GLPS) found explicit examples of $G_2$ and $Spin(7)$ holonomy metrics on bundles over $\mathbb{T}^4$. Both of these examples admit K\"ahler reductions as described in section \ref{kahlerreduction}: the $G_2$ example being a special case of $s=1$ cf. \cite{Apostolov2003, ChiossiSalamonIntrinsicTorsion}.
Their $Spin(7)$ example arises from Theorem \ref{constantthm} when $M=\mathbb{T}^4$ with $c=0$, $A=1$ and $(a,b,p,q)=(0,0,0,1)$. Choosing different integers $A$ amounts to pulling back their $Spin(7)$ $4$-form $\Phi$ to covers of the circle bundle determined by $d\eta$. The induced $Spin(7)$ metric is then rescaled by the factor $A^{1/2}$ on the covering space. We describe below these examples, as well as certain related Calabi-Yau metrics, in more detail . An interesting feature in all the following examples is that the symplectic form on $P^6$ is always the same but the complex structure (on the $\C^\times$ fibre) changes. Put differently, this means that each example below corresponds to a different integrable section of the associated bundle on $(P^6,\om)$ with fibre $\frac{Sp(6,\R)}{SU(3)}$.\\

\noindent\textbf{Spin(7): }Let $P^6=Q^5 \times \R_t$, where $Q^5$ is a nilmanifold whose Lie algebra is given by
\[(0,0,0,0,13+42) \]
in Salamon's notation \cite{Salamoncomplexstructures}. We remind the reader that a nilmanifold is the quotient of a simply connected nilpotent Lie group by a cocompact lattice whose existence in this case is guaranteed by Malcev's criterion cf. \cite{latticeliegpsraghunathan2012}. So we can choose a coframing $e^i$ on $Q^5$ satisfying 
\begin{gather*}
de^5=e^{13}+e^{42},\\
de^i=0, \ \text{ for\ } i=1,2,3,4
\end{gather*}
where $e^{ij}:=e^i \w e^j$. We define a K\"ahler $SU(3)$-structure on $P^6$ by
\begin{gather*}
\om = d(t\cdot e^5),\\
\Om= t(\sigma_3 +i \sigma_1)\w (-t^{-2}e^5+it^2dt),
\end{gather*}
where $\sigma_1:=e^{12}+e^{34}$, $\sigma_2:=e^{13}+e^{42}$ and $\sigma_3:=e^{14}+e^{23}$ denote the standard self-dual $2$-forms on $\mathbb{T}^4$.
The torsion forms, cf. Lemma \ref{torsionlemma}, are then given by
\begin{gather*}
d\Om^+ = -t^{-1}dt\w \Om^+,\\
d\Om^- = t^{-5}e^5 \w \Om^+.
\end{gather*}
Taking $H=t^{4/3}$ and $s=t$, we have $d\xi=\sigma_3$ and $d\eta=\sigma_1$. Hence from Proposition \ref{spin7thm} it follows that $\Phi$ is torsion free and the induced metric is given by
\[g_\Phi=t^6dt^2+t^{-2}(\xi^2 + \eta^2 + (e^5)^2)+t^3((e^1)^2+(e^2)^2+(e^3)^2+(e^4)^2).\] 
In fact one can verify that the holonomy group is \textit{equal} to $Spin(7)$, using \textsc{Maple} for instance. This can be done by verifying that the dimension of the holonomy algebra, or equivalently by the Ambrose-Singer Theorem the rank of the curvature operator, is equal to $21$. 

\noindent\textbf{G$_2$: }If we keep $\om$ on $P^6$ unchanged but modify the complex structure so that
\[\Om= t(\sigma_3 +i \sigma_1)\w (-t^{-3/2}e^5+it^{3/2}dt) \]
and take $H=t$ and $s=1$, then from Proposition \ref{g2thm} we see that $\vp$ is torsion free. The induced metric is given by
\[g_\vp=t^4dt^2+t^{-2}(\xi^2  + (e^5)^2)+t^2((e^1)^2+(e^2)^2+(e^3)^2+(e^4)^2).\] 
Here $d\xi$ is again defined as in the $Spin(7)$ example. One can once again verify that the holonomy group of $g_\vp$ is \textit{equal} to $G_2$ cf. \cite{Apostolov2003, GibbonsKahler}.

\noindent\textbf{CY: }Following the same strategy, it is easy to see that we obtain a torsion free $SU(3)$-structure by keeping $\om$ unchanged and taking
\[ \Om =t(\sigma_3 +i \sigma_1)\w (-t^{-1}e^5+it^{1}dt).  \]
The induced metric is given by 
\[g_\om=t^2dt^2+t^{-2} (e^5)^2+t((e^1)^2+(e^2)^2+(e^3)^2+(e^4)^2).\] 
Note that from \cite[theorem $3.6$]{UdhavFowdar}, we can also construct a metric with holonomy $SU(4)$ from this Calabi-Yau $3$-fold. The $SU(4)$-structure on $\hat{N}^8$ is given by
\begin{gather*}
	\hat{\om}=s^{2/3}\om + \hat{\eta} \w d(s^{2/3}),\\
	\hat{\Om}=\Om \w (-\hat{\eta} - i \frac{2}{3}s^{5/3}ds ),
\end{gather*}
with $d\hat{\eta}=-\om$. Topologically $\hat{N}^8=\hat{L}^7 \times \R_s^+$, where $\hat{L}^7$ is the $S^1$ bundle determined by $[-\om]\in H^2(P^6,\mathbb{Z})$. This gives an example of a cohomogeneity two Einstein metric. Explicitly it is given by
\[ \hat{g}=s^{2/3}(t^2  dt^2+t^{-2}(e^5)^2+tg_{\mathbb{T}^4})+s^{-2}\hat{\eta}^2+(\frac{2}{3}s^{2/3}ds)^2 .\]
By analogy to our construction, this can also be viewed as an `inversion' of the K\"ahler reduction, from a CY $3$-fold to a CY $4$-fold, where the moment map is $s^{2/3}$. 
\subsection{$Spin(7)$ metrics from Gibbons-Hawking ansatz}\label{GibbonsHawkingAnsatz}
It is clear that one can replace $\mathbb{T}^4$ by any hyperK\"ahler manifold $M^4$ in all the above examples. Although it is not generally true that the triple of hyperK\"ahler forms define integral cohomology classes this is nonetheless always true locally. Thus, combined with the Gibbons-Hawking ansatz this gives infinitely many local examples of metrics with holonomy $Spin(7)$ starting from just a positive harmonic function on an open set in $\mathbb{R}^3$. 

More precisely, given an open set $B \subset \R^3$ with the Euclidean metric and coordinates $(x,y,z)$, together with a positive harmonic function $V:B\to \R^+$ satisfying the integrality condition $[-*dV]\in H^2(B,\mathbb{Z})$, then we can define a hyperK\"ahler triple on the total space $M^4$ of  the circle bundle by
\begin{align*}
\om_1=\theta \w dx + V\ dy \w dz, \\
\om_2=\theta \w dy + V\ dz \w dx,\\
\om_3=\theta \w dz + V\ dx \w dy,
\end{align*}
where $\theta$ is a connection $1$-form satisfying $d\theta=-*dV$.

\begin{Cor}\label{GibbonsHawkingCor1}
Given a hyperK\"ahler four-manifold $(M^4,g_M, \om_1, \om_2, \om_3)$ such that $[\om_1],[-\om_2],[-\om_3] \in H^2(M,\mathbb{Z})$, let $K^7$ denote the total space of the $\mathbb{T}^3$ bundle determined by the latter triple. Then we can define a metric with holonomy contained in $Spin(7)$ on $K^7 \times \R^+_s$ by
\begin{align}
g_\Phi=s^{-2}\eta^2&+(s+c)^{-2}\xi^2+(s+p)^{-2}\alpha^2\label{gh}\\&+s^2(s+c)^2(s+p)^2\ ds^2+{s(s+c)(s+p)}g_{M}\nonumber ,\end{align}
where constants $c$, $p \in [0,+\infty)$ and the connection $1$-forms $\alpha$, $\eta$, $\xi$ satisfy
 $$(d\alpha, d\xi,d\eta)= (\om_1,-\om_2,- \om_3).$$
 Moreover, if $M^4$ admits a triholomorphic $S^1$ action then we can locally write
 \[g_M=V^{-1}\theta^2+V(dx^2+dy^2+dz^2)  \]
via the Gibbons-Hawking ansatz and hence $g_\Phi$ is completely determined by $V$.
\end{Cor}
\noindent The metric (\ref{gh}) corresponds to constant solutions with $A=1$, $(a,b,q)=(0,0,1)$ and is defined for $s\in (0,+\infty)$. This metric is incomplete as $s\to 0$ since the circle fibre corresponding to the connection form $\eta$ always blows up while the length of the other two circles fibres converge to $c^{-1}$ and $p^{-1}$ as $s\to 0$. It is not hard to see that $g_\Phi$ is complete as $s \to \infty$. The family of metrics given by (\ref{gh}), especially in the case when $M^4$ is $\mathbb{T}^4$ or a suitable $K3$ surface, might be useful in future gluing problems as in the hyperK\"ahler case recently investigated in \cite{hsvz}.\\

\noindent\textbf{A remark on the `generalised Calabi ansatz.'}
The $SU(3)$ and $SU(4)$ holonomy metrics appearing in this section in fact arise from a special case of the Calabi construction \cite{Calabi1957}. In our setting this can be neatly described as follows: given a Calabi-Yau $n$-fold $\hat{N}^{2n}$ with symplectic form $\sigma$ and holomorphic volume form $\Psi$ we define a connection $1$-form $\gamma$ on the line bundle $L_{\hat{N}}$ with Chern class determined by
\[ d\gamma=-\sigma .\]
We then obtain a torsion free $SU(n+1)$-structure on an open set of the total space $L_{\hat{N}}$ given by
\begin{gather*}
	\hat{\sigma}=-d(r^2 \gamma),\\
	\hat{\Psi}= \Psi \w ( \gamma + i \frac{2}{n+2} d(r^{n+2})),
\end{gather*}
where $r$ denotes a radial coordinate from the zero section. The examples above can thus be interpreted as a `\textit{generalised Calabi ansatz}' for exceptional holonomy metrics, whereby one uses the hyperK\"ahler forms $\om_1,\om_2$ and $\om_3$ in succession to construct $SU(3)$, $G_2$ and $Spin(7)$ holonomy metrics. The relations between the Einstein manifolds described above can be expressed in the diagram below. 
\begin{center}
\begin{tikzcd}[column sep=tiny]
& (N^8,\Phi) \ar[dr,"S^1"] \ar[dl,"\mathbb{T}^2" '] \ar[dd,"\mathbb{T}^3\times \R^+"]
&
&[1.5em] \\
(P^6,\omega,\Omega) \ar[dr,"\C^\times" ']
&
& (L^7,\varphi) \ar[dl,"\mathbb{T}^2\times \R^+"] \\
& (M^4,\omega_1,\omega_2,\omega_3) 
&
\end{tikzcd}
\end{center}
The submersions of $P^6$ and $L^7$ with their Ricci flat metric to $M^4$ correspond to the Calabi ansatz and Apostolov-Salamon construction, respectively. 

By contrast to the above examples, where all the circle bundles were determined by the hyperK\"ahler forms (since we had $d\al=\sigma_2$), in the next section we give examples corresponding to the case when $b \neq 0$. 
\section{More examples}\label{example2}
The $G_2$ example we describe in this section has also appeared in \cite{Apostolov2003, ChiossiSalamonIntrinsicTorsion} but the $Spin(7)$ metric does not seem to have been mentioned in the literature. 

Let $P^6=Q^5 \times \R_t$, where $Q^5$ is a nilmanifold whose Lie algebra is given by
\[(0,0,0,0,24) .\]
So $Q^5$ is again topologically a circle bundle over $\mathbb{T}^4.$ We define a K\"ahler $SU(3)$-structure on $P^6$ by
\begin{gather*}
\om = e^{13}-d(t^2 e^5),\\
\Om= t(-\sigma_1+i\sigma_3 )\w (-2t^{3}dt+it^{-2}e^5),
\end{gather*}
where $\sigma_i$ denote the standard self-dual $2$-forms as before.
The torsion forms are given by
\begin{align*}
d\Om^+ &= -t^{-1}dt\w \Om^+,\\
d\Om^- &= -\frac{1}{2}t^{-6}e^5 \w \Om^+.
\end{align*}
Taking $H=t^2$ so that $d\xi=-\sigma_3$, one can verify directly that the hypothesis of Proposition \ref{g2thm} are satisfied. We thus get a holonomy $G_2$ metric as described in \cite{Apostolov2003}. This was also shown to arise from the Hitchin flow of half-flat $SU(3)$-structures in \cite{ChiossiSalamonIntrinsicTorsion}. 

Following the same strategy as in the GLPS examples of last section we keep $\om$ unchanged and consider
\[\Om= t(-\sigma_1+i\sigma_3 )\w (-2t^{4}dt+it^{-3}e^5) .\]
Taking $H=t^{8/3}$ we see that the hypothesis of Proposition \ref{spin7thm} are satisfied, with $d\eta=-\sigma_1$ and $d\xi=-\sigma_3$. Thus we get a $Spin(7)$ metric explicitly given by
\[g_\Phi= (t^2e^1)^2 + (t^3e^2)^2+(t^2e^3)^2+(t^3e^4)^2+(t^{-1}e^5)^2+(t^{-2}\eta)^2+(t^{-2}\xi)^2+4t^{12}dt^2 .\]
Again one verifies by computing the dimension of the holonomy algebra that the holonomy is indeed \textit{equal} to $Spin(7)$.
Of course, we can also construct holonomy $SU(3)$ and $SU(4)$ metrics by carrying out an analogous argument as in the previous section.
\subsection{$Spin(7)$ metrics from Tod's ansatz}\label{GHansatz}
As in section \ref{example1} there is a natural way of obtaining many local examples using the so-called Tod's ansatz cf. \cite[Prop. $3.1$]{ArmstrongTodansatz}. The idea is again to apply the Gibbons-Hawking ansatz but choosing the harmonic function to depend on only two variables. 

In the notation of subsection \ref{GibbonsHawkingAnsatz}, if we choose $V:B\to \R^+$, independent of say coordinate $x$, then in addition to the Gibbons-Hawking hyperK\"ahler triple, we can also define an almost K\"ahler form by
\[\om_0=\theta \w dx - V\ dy \w dz .\]
 Thus, we can again appeal to Theorem \ref{constantthm} to construct $Spin(7)$ metrics. In particular, for $A=1$ and $(a,b,q)=(0,1,1)$ we have:
\begin{Cor}\label{GibbonsHawkingCor2}
Given a hyperK\"ahler four-manifold $(M^4,g_M, \om_1, \om_2, \om_3)$ together with an almost K\"ahler form $\om_0$ compatible with the opposite orientation such that $[\om_0+\om_1],[-\om_2],[-\om_3] \in H^2(M,\mathbb{Z})$, let $K^7$ denote the total space of the $\mathbb{T}^3$ bundle determined by the latter triple. Then we can define a metric with holonomy contained in $Spin(7)$ on $K^7 \times \R^+_s$ by
\begin{align*}
g_\Phi=s^{-2}\eta^2&+(s+c)^{-2}\xi^2+p^{-1}(2s+p)^{-1}\alpha^2\\&+ps^2(s+c)^2(2s+p)\ ds^2+{s(s+c)}g_{\tilde{\om}_1} ,\end{align*}
where constants $c$, $p \in (0,+\infty)$, $g_{\tilde{\om}_1}$ is defined by (\ref{symplectic}) and the connection $1$-forms $\alpha$, $\eta$, $\xi$ satisfy
$$(d\alpha, d\xi,d\eta)= (\om_0+\om_1,-\om_2,- \om_3).$$
Moreover, if $M^4$ admits a triholomorphic $S^1$ action then $g_\Phi$ is completely determined by the harmonic function $V(y,z)$, as in Tod's ansatz.  
\end{Cor}

\section{Hypersurfaces and Hitchin flow}\label{hypersurfaces}

In this section we explain how the aforementioned metrics may also be obtained by evolving suitable cocalibrated $G_2$-structures. It is well-known that an oriented hypersurface $\tilde{L}$ in a $Spin(7)$ manifold $(N^8,\Phi)$ inherits a cocalibrated $G_2$-structure defined by
\[ \phi= \mathbf{n} \lrcorner \ \Phi ,\]
where $\mathbf{n}$ denotes the unit normal vector field. As a converse Hitchin shows that given a cocalibrated $G_2$-structure $\phi_0$ on a compact seven-manifold $\tilde{L}$ one can define a torsion free $Spin(7)$-structure on $N=\tilde{L}\times (0,T)$ by solving the system
\begin{gather}
d_{\tilde{L}}(*_{\phi_t}\phi_t)=0,\label{hit1}\\
\frac{\partial}{\partial t}(*_{\phi_t}\phi_t)=d_{\tilde{L}}\phi_t,\label{hit2}
\end{gather}
where $t\in (0,T)$, \cite[Theorem 7]{Hitchin01stableforms}. Furthermore, Bryant shows that if $\phi_0$ is real analytic then there always exists a local solution to (\ref{hit1}) and (\ref{hit2}), cf.  \cite[Theorem 7]{BryantEmbedding}. The resulting $Spin(7)$ form on $N^8$ is then given by
\[\Phi=dt \w \phi_t + *_{\phi_t}\phi_t. \]
There is also an analogous theory for oriented hypersurfaces in $G_2$ manifolds \cite[Theorem 8]{Hitchin01stableforms}. In fact the $G_2$ holonomy metrics appearing in the last two sections have been described via this technique in \cite{ChiossiSalamonIntrinsicTorsion}. 

We shall now explain how the $Spin(7)$ examples corresponding to constant solutions may also be obtained via the Hitchin flow. From the definition of $\alpha$ and the expression relating the metrics $g_\Phi$ and $g_\om$, it is straightforward to compute
\[\|dH \|_{g_\Phi}= \frac{A}{u^{1/2}H^{1/2}s^{1/3}s'} .\]
Thus, from the expression for $g_\Phi$ in Corollary \ref{spin7cor} we see that one can define a geodesic coordinate $t$ on $N^8$  by
\[t=\frac{1}{A^{3/2}} \int s(s+c)( (p+qs)^2-(a+bs)^2  )^{1/2}ds. \]
The hypersurfaces $\tilde{L}_t$ in $N^8$ corresponding to level sets of $t$ are the $\mathbb{T}^3_{\al,\xi,\eta}$ bundles over $M^4$ defined by (\ref{curvatureforms}) and are endowed with cocalibrated $G_2$-structures $\phi_t$. From our expression for $\Phi$ we have that
\[*_{\phi_t}\phi_t=\eta \w \xi \w \tilde{\om}_1+\Big(\frac{s+c}{A}\Big)\eta\w\alpha \w \om_2+\frac{s^{2}(s+c)^2}{2A^2}\tilde{\om}_1\w \tilde{\om}_1 + s\alpha \w \xi \w \om_3 .  \]
It is easy to see, from the expressions of the curvature forms (\ref{curvatureforms}), that (\ref{hit1}) holds. We leave it to the interested reader to verify that (\ref{hit2}) is also satisfied.

For instance, in the GLPS $Spin(7)$ example we find that $4t=H^3=s^4$ and an orthonormal coframing for $\phi_t$ is given by
\[ s^{3/2}e^{1},\ \ s^{3/2}e^{2},\ \ s^{3/2}e^{3},\ \ s^{3/2}e^{4},\ \ s^{-1}\eta,\ \ s^{-1}\xi,\ \ s^{-1}\al . \]
\begin{Rem}
Although there are many cocalibrated $G_2$-structures on nilmanifolds, the scarcity of explicit metrics with holonomy \textit{equal} to $Spin(7)$ stems from the fact that the Hitchin flow is generally hard to solve and moreover, it often leads to $SU(4)$ holonomy metrics rather than $Spin(7)$ cf. \cite{MarcoFreibert2018}.
\end{Rem}

\section{Perturbation of constant solutions}\label{nonconstant}

In this section we describe explicit solutions to Corollary \ref{spin7cor} which vary on $M^4$ i.e. with $d_Mu \neq 0$. Our solutions are obtained by perturbing the K\"ahler potential of the constant solution examples. We shall again assume that $(M^4,\om_1,\om_2,\om_3)$ is a hyperK\"ahler manifold together with an anti-self-dual $2$-form $\om_0$ as in section \ref{constant}. We look for solutions to  (\ref{equ1}) and (\ref{equ2}) with $\tilde{\om}_1$ of the form
\[ \tilde{\om}_1=(a+bs)\om_0+ (p+qs)\om_1 + d_Md_M^c G.\]
When $G=0$ we recover the constant solutions metrics. We also know from the global $dd^c$ lemma that any K\"ahler form in the same cohomology class can be expressed in this form for some function $G$. Equation (\ref{equ2}) can now be written as
\begin{equation}
d_Md_M^c(u-A^2\ddot{G})=0,\label{equ2new} \end{equation}
and condition (\ref{equ1}) becomes
\begin{align}\label{equ1new}
u\cdot \om_1^2 = \frac{s(s+c)}{A}(\Big(&(p+qs)^2-(a+bs)^2 +(p+qs) \Delta_M G\Big)\cdot \om_1^2\\&+2(a+bs)(d_Md_M^cG)\w\om_0 + (d_M d_M^c G)^2),  \nonumber
\end{align}
where $\Delta_M$ denotes the Hodge Laplacian on $(M^4,g_{\om_1},J_1)$. Note that we can also express the last term as
\[(d_Md_M^cG)^2=(\frac{1}{4}(\Delta_MG)^2-\frac{1}{2}\|(d_Md_M^cG)_0\|^2_{g_{\om_1}})\cdot \om_{1}^2, \]
where $(d_Md_M^cG)_0$ denotes the traceless component of $d_Md_M^cG$ or equivalently its projection in $\Lm^2_-$. The system (\ref{equ1new}) and (\ref{equ2new}) is still quite hard to solve in full generality, so we shall make some further simplifying assumptions.

From \cite[Theorem 2.4, 3.2]{Bedford77} we know that a smooth real function $F$ on $M^4$ satisfies 
\[(d_Md^c_MF)^2=0 \] 
if and only if $M^4$ admits a foliation by complex submanifolds, with the leaves corresponding to the integral (complex) curves of the ideal generated by $d_Md^c_MF$. In this case we may assume there exists locally a fibration $\pi:M^4 \to \Sigma^2 $, where $\Sigma^2$ is a complex curve and that $F$ descends to a function on $\Sigma^2$. Under this hypothesis on $G$, for each $s$, we can eliminate the quadratic term in (\ref{equ1new}) .

We shall now illustrate how one can construct metrics with holonomy \textit{equal} to $Spin(7)$ under these assumptions by perturbing the GLPS example. \\

\noindent\textbf{Example.}
As before, consider $M=\mathbb{T}^4$ with local coordinates $(x_1,x_2,x_3,x_4)$ and endowed with the standard flat hyperK\"ahler structure. We set $(a,b,p,q)=(0,0,0,1)$, $A=1$ and consider $G$ of the form:
\[ G(s,x_1,x_2)=v(s)\cdot F(x_1,x_2)+\frac{1}{12}s^4 .\]
$\Sigma^2$ here is the elliptic curve $\mathbb{T}^2$ with coordinates $(x_1,x_2)$. Defining $u$ by
\[ u=\ddot{G} \]
automatically solves (\ref{equ2new}), and (\ref{equ1new}) becomes equivalent to the pair:
\begin{gather*}
	\Delta_MF=\mu F,\\
	\ddot{v}=\mu s^2 v,
\end{gather*}
where $\mu$ is a constant. The reader might recognise that the second equation is the well-known Weber equation. With $\mu=1$, a simple solution is given by
\begin{gather*}
	F=\sin(x_1),\\
	v=U(0,\sqrt{2}s),
\end{gather*}
with $U(a,t)$ denoting the parabolic cylinder function, see \cite[Chapter 19]{HandbookofMathematicalFunctions} for the definition in terms of hypergeometric functions. From Corollary \ref{spin7cor} we find that the connection form $\alpha$ is can be expressed as
\[ \alpha= dx_5-\dot{v}\cos(x_1)dx_2,  \]
where $x_5$ denotes the angular coordinate on the $S^1$ fibre. One can verify that $g_\Phi$, well-defined for $\{s\ | \ U(0,\sqrt{2}s)<1\}$, has holonomy \textit{equal} to $Spin(7)$. Thus this gives a $Spin(7)$ perturbation of the GLPS metric. 

Setting $f(x_1,s)=1+\sin(x_1)v(s)$ and denoting by $(x_6,x_7)$ the coordinates on the $\mathbb{T}^2$ fibres, we can express the connection forms as
\begin{gather*}
\xi=dx_6-x_3dx_1-x_2dx_4,\\
\eta=dx_7-x_4dx_1-x_3dx_2.
\end{gather*}
and hence, the perturbed metric can be expressed in local coordinates as
\[ g_\Phi=s^2(f(dx_1^2+dx_2^2)+dx_3^2+dx_4^2)+ f^{-1}\al^2+s^{-2}(\xi^2+\eta^2)+s^{4}fds^2.\]
One can get other similar examples by choosing $\mu=-1$ and allowing $F$ to depend on both $x_1$ and $x_2$ for instance. 
\begin{Rem}
Another source of compact examples fitting in the above construction are elliptic $K3$ surfaces. These examples however require more sophisticated tools to study as the metrics are no longer explicit.
\end{Rem}

 \noindent We conclude our study of the $S^1$ K\"ahler reduction and now proceed to the $\mathbb{T}^2$ case.
\section{Further reduction \textrm{II}}\label{thirdreduction}
\subsection{$\mathbb{T}^2$ K\"ahler reduction}
Recall from subsection \ref{secondreduction} that there are two natural constraints to impose on the function $s$. The first is that $s$ is a function of $H$, and the second that $s$ and $y:=Hs^{-1/3}$ are independent functions on $P^6$ with $U$ and $W$ orthogonal. Having investigated the former situation, we shall now study the latter case and construct yet more examples of $Spin(7)$ metrics.

We follow the same strategy as in the proof of Theorem \ref{generalAS}. We first define connection $1$-forms $\alpha$ and $\kappa$ on $P$ by
\begin{align}
	\al(\cdot)&=g_{\om}(U,\cdot\ )u,\\
	\kappa(\cdot)&=g_{\om}(W,\cdot\ )w,
\end{align}
where $u:=\|U\|^{-2}_{\om}$ and $w:=\|W\|^{-2}_{\om}$. From our assumptions on $U$ and $W$, it is easy to see that they commute and that they are infinitesimal symmetries of the $SU(3)$-structure. Hence they define a $(\C^\times)^2$ action on $P^6$ and we can once again carry out a K\"ahler reduction:
\[(P^6,\om,\Om,J)\xrightarrow{/\!\!/ \mathbb{T}^2} (\Sigma^2,\tilde{\om},\Upsilon,\tilde{J}).  \]
The holomorphic $(1,0)$-form $\Upsilon$ on $\Sigma^2$ is defined by
\[ \Upsilon_1 -i \Upsilon_2:=\frac{1}{4}(W-iJW)\lrcorner (U-iJU)\lrcorner (H^{1/2}s^{1/3}\Om) \]
and the quotient symplectic form $\tilde{\om}(s,y)$ is given by
\[\om=-\alpha \w dy+\kappa\w ds+\tilde{\om}(s,y). \]
Note that if $\Sigma^2$ is compact then  it must be an elliptic curve, since it has trivial first Chern class. Unlike in the $\C^\times$ case however the horizontal lifts of $U$ and $W$ do not preserve the $Spin(7)$-structure as
\[\mathcal{L}_U\eta=\Upsilon_2 \text{\ \ \ \ and\ \ \ \ } \mathcal{L}_W\xi=-\Upsilon_1.\]
Hence, for each fixed $s$ and $y$, the six dimensional submanifold in $N^8$ corresponds to a $\mathbb{T}^2$ bundle over a $\mathbb{T}^2$ bundle over the surface $\Sigma^2.$ In the case when $\Sigma=\mathbb{T}^2$, this submanifold is just a nilmanifold. Thus, we shall generally refer to these submanifolds as `nilbundles'.\\

From (\ref{connection1}) and (\ref{connection2})  we can equivalently write $\al $ and $\kappa$ as
\begin{gather}
	\al=-ud^cy,\\
	\kappa=wd^cs.
\end{gather}

As in subsection \ref{S1Kahkerreduction}, we can once again express the data $(P^6,\om,\Om,\al,\kappa)$ purely in terms of  $(\Sigma^2,\tilde{\om},\Upsilon, u(s,y),w(s,y))$, and thus provide a way to invert the K\"ahler reduction. The proof follows the same strategy as in the previous case. The result is summed up as follows:
\begin{Th}\label{secondmaintheorem}
Given a complex curve $(\Sigma^2,\tilde{J})$ with a holomorphic $(1,0)$-form $\Upsilon_1-i\Upsilon_2$, two $1$-parameter families of positive functions $u=u(y)$ and $w=w(s)$, and a family of K\"ahler forms $\tilde{\om}(s,y)$ satisfying
\begin{align}
-(s\cdot y) \ \tilde{\om}&=(u\cdot w) \Upsilon_1 \w \Upsilon_2,\label{equation1}\\
	\frac{\partial^2 \tilde{\om}}{\partial y^2}&=d_\Sigma d^c_\Sigma u,\label{equation2}\\
	\frac{\partial^2 \tilde{\om}}{\partial s^2}&=d_\Sigma d^c_\Sigma w,\label{equation3}\\
	\frac{\partial^2 \tilde{\om}}{\partial y \partial s}&=0\label{equation4},
\end{align}
where $d_\Sigma$ denotes the exterior differential on $\Sigma^2$ and $d^c_\Sigma:=\tilde{J}\circ d_{\Sigma}$, and for $(s,y)\in I^+_s \times I^+_y\subset \R^+ \times \R^+$. Then there exists, on the `nilbundle' over $\Sigma^2\times I_s^+ \times I_y^+$, defined by the curvature $2$-forms:
\begin{align*}
	d\alpha&=-d^c_\Sigma u\w dy+\frac{\partial \tilde{\om}}{\partial y},\\
	d\kappa&=d^c_\Sigma w\w ds-\frac{\partial \tilde{\om}}{\partial s},\\
	d\xi&=\Upsilon_1\w \kappa  + \Upsilon_2 \w  w\ ds,\\
	d\eta&=\al \w \Upsilon_2+u\ dy\w \Upsilon_1,
\end{align*}
a torsion free $Spin(7)$-structure $\Phi$ inducing the metric:
\begin{equation}
g_{\Phi}=s^{-2}\eta^2+y^{-2}\xi^2+y\cdot s\ (u^{-1}\al^2+u\ dy^2+w^{-1} \kappa^2+w\ ds^2+g_{\tilde{\om}}), \label{metric} 
\end{equation}
where $g_{\tilde{\om}}$ denotes the K\"ahler metric on $(\Sigma^2,\tilde{J})$ determined by $\tilde{\om}.$
\end{Th}
\subsection{A general existence result}\label{generalstrategy}
Before constructing explicit examples we first describe how to find a general solution to Theorem \ref{secondmaintheorem}.

We pick complex coordinate $z=x_1+ix_2$ on $\Sigma^2$ so that we can write $\Upsilon=dx_1+idx_2$ and the K\"ahler form is given by
\[\tilde{\om}=F(y,s)\ dx_1\w dx_2, \]
where $F$ is a positive function on $\Sigma^2$ depending on $y$ and $s$. From equation (\ref{equation4}) we have that
\[F(s,y)=F_1(y)+F_2(s) .\]
Thus, equations (\ref{equation2}) and (\ref{equation3}) are equivalent to the pair:
\begin{gather}
	\frac{\partial^2 F_1}{\partial y^2}=-(u_{x_1,x_1}+u_{x_2,x_2}),\label{dude1}\\
		\frac{\partial^2 F_2}{\partial s^2}=-(w_{x_1,x_1}+w_{x_2,x_2}),\label{dude2}
\end{gather}
while equation (\ref{equation1}) reduces to 
\begin{equation}
s\cdot y\ (F_1(y)+F_2(s))=u(y)\cdot w(s).\label{dude3}
\end{equation}
It follows, without loss of generality, that either $F_1$ or $F_2$ must be zero and hence, that either $u(y)$ or $w(s)$ is a $1$-parameter family of harmonic functions on $\Sigma^2$. In particular if $\Sigma=\mathbb{T}^2$ then either $u$ or $w$ is constant.

Assuming that $F_2=0$, (\ref{dude2}) and (\ref{dude3}) implies that
\[ F_1(y)=\bigg(\frac{u(y)}{y}  \bigg)\cdot G(x_1,x_2) \text{\ \ \ \ and\ \ \ \ } w(s)=s\cdot G(x_1,x_2),\]
for a positive harmonic function $G:\Sigma\to \R^+$, \textit{independent} of $s$ and $y$. Therefore, solving the general system of Theorem \ref{secondmaintheorem} amounts to solving the single PDE 
\begin{equation}
G\cdot\frac{\partial^2}{\partial y^2}(\tilde{u}(y))=y\cdot\Delta_\Sigma \tilde{u}(y), \label{dude4}
\end{equation}
where $\tilde{u}(y):=\frac{u(y)}{y}$ and $\Delta_\Sigma$ denotes the Hodge Laplacian on $\Sigma^2$. Given real analytic initial data we can once again appeal to the Cauchy-Kovalevskaya theorem for the existence and uniqueness of a real analytic solution.
\begin{Cor}
	Given real analytic functions $u_{0}$ and $u_1$ on (an open set of) a complex curve $(\Sigma^2,J_{1},\Upsilon_1-i\Upsilon_2)$ with $u_0>0$, then there exists a unique real analytic solution $\tilde{u}(y)$, for $y$ in a small interval, to (\ref{dude4}) with $\tilde{u}(0)=u_{0}$ and $\frac{\partial\tilde{ u}}{\partial y}(0)=u_1$, and hence by Theorem \ref{secondmaintheorem} a  torsion free $Spin(7)$-structure.
\end{Cor}
\begin{Rem}
If we look for separable solutions $\tilde{u}=A(y)\cdot B(x_1,x_2)$, then (\ref{dude4}) becomes equivalent to the pair
\begin{gather}
\frac{\partial^2}{\partial y^2}A(y)=\mu\cdot y\cdot A(y),\label{Airy}\\
\Delta_\Sigma B=\mu \cdot G \cdot B,\label{nonairy}
\end{gather}
where $\mu$ is a constant and equation (\ref{Airy}) is the well-known Airy equation for $\mu \neq 0$.
\end{Rem}
  
In summary, we have reduced the problem of finding $Spin(7)$
metrics admitting  K\"ahler reduction with $\mathbb{T}^2$ symmetry to choosing a positive harmonic function $G$ and solving (\ref{dude4}). We now proceed to describe some explicit examples.

\section{Constant solutions $\mathrm{II}$}\label{constant2}
In this section we describe the simplest solutions which arise when $u$ and $w$ are both constants on $\Sigma^2$. Without loss of generality, this corresponds to setting $\mu=0$, $B=1$ and $G=c$ is a positive constant, in (\ref{Airy}) and (\ref{nonairy}). The general solution is then given by
\begin{gather*}
w(s)=c s,\\
u(y)=y (p+qy),\\
\tilde{\om}=c(p+qy)\ dx_{12},
\end{gather*}
where $p,q\in \R$ and the positivity condition on $u$ implies that the solution is valid for $p+qy>0$.

Denoting the coordinates on the fibres by $(x_3,x_4,x_5,x_6)$, we can express the connection $1$-forms as
\begin{gather*}
\al=cq x_1 dx_{2}+dx_3,\\
\kappa=dx_4,\\
\xi=dx_5+x_1dx_4-csx_2ds,\\
\eta=dx_6+x_2 dx_3-yx_1(p+qy)dy.
\end{gather*}
If we fix $y$ and $s$, then we have that
\[(d\al,d\kappa,d\xi,d\eta)=(cq dx_{12}, 0 ,dx_{14},dx_{23}). \]
Thus, it follows that if $\Sigma=\mathbb{T}^2$ then these codimension $2$ submanifolds are diffeomorphic to nilmanifolds with nilpotent Lie algebra isomorphic to either
\[ (0,0,0,0,12,34) \text{\ \ \ \ or \ \ \ \ } (0,0,0,12,13,24),\]
depending on whether $q$ is zero or not. The former corresponds to the $2$-step nilpotent Lie algebra of the product of two real Heisenberg groups while the latter corresponds to an indecomposable $3$-step nilpotent Lie algebra.

One can compute the dimension of the holonomy algebra and verify that the corresponding metrics determined by expression (\ref{metric}):
\begin{align*}
g_\Phi=\ &s^{-2}\eta^2+y^{-2}\xi^2+s(p+qs)^{-1}\al^2+c^{-1}y\kappa^2\\
&+y^2s(p+qs)dy^2+cys^2ds^2+csy(p+qy)(dx_1^2+dx_2^2),
\end{align*}
have holonomy \textit{equal} to $Spin(7)$. Thus, this classifies the constant solutions examples. We shall now consider some non-constant solutions. 

\section{Examples of non-constant solutions }\label{nonconstant2}
In this section we give explicit examples of $Spin(7)$ metrics which vary on $\Sigma^2$. To illustrate the different cases that can arise from our construction, in the first example we consider a non-compact surface so that we may choose non-constant harmonic functions on $\Sigma^2$ and in the second example we consider a separable solution with $\mu=1$ on $\mathbb{T}^2$. As in the previous section we shall denote the fibre coordinates by $(x_3,x_4,x_5,x_6)$.\\

\noindent\textbf{Example 1.} We take $\Sigma=\C-B_1(0)$, where $B_1(0)$ denotes the unit ball centred at the origin,  with the holomorphic form $\Upsilon=dx_1+idx_2$ as before.
Following the strategy outline in subsection \ref{generalstrategy}
we find that a solution is given by choosing $F_1(y)=y\ln(r)$, $F_2(s)=0$, $w(s)=s\ln(r)$ and $u(y)=y^2 $, where $r:=x_1^2+x_2^2$. The connection $1$-forms are given in local coordinates by:
\begin{gather*}
\al=dx_3+(x_1\ln(r)-2 x_1+2x_2 \arctan\Big(\frac{x_1}{x_2}\Big))dx_2,\\
\kappa=dx_4-\frac{1}{2}s^2d^c_\Sigma\ln (r),\\
\xi=dx_5+x_1dx_4+\frac{1}{2}s^2\ln (r)dx_2\\
\eta=dx_6+x_2dx_3-x_1y^2dy.
\end{gather*}
One can again check that the induced metric has holonomy \textit{equal} to $Spin(7)$. \\

\noindent\textbf{Example 2.} We now take $\Sigma=\mathbb{T}^2$ endowed with the standard flat K\"ahler structure. With $\mu=1$ the general solution to (\ref{Airy}) is the Airy function $\mathrm{Ai}(y)$. Thus, picking $F_1(y)=\mathrm{Ai}(y)\sin(x_1)$, $F_2(s)=0$, $w(s)=s$ and $u(y)=y\mathrm{Ai}(y)\sin(x_1)$, we obtain another solution. The connection $1$-forms are now given by:
\begin{gather*}
\al=dx_3-\mathrm{Ai}'(y)\cos(x_1)dx_{2},\\
\kappa=dx_4,\\
\xi=dx_5+x_1dx_4-sx_2ds,\\
\eta=dx_6+x_2dx_3+y\mathrm{Ai}(y)\cos(x_1)dy.
\end{gather*}
The resulting $Spin(7)$ metric is well-defined on the set where $u>0$. By taking  $\Sigma=\C$ and $\mu=-1$ instead we find yet more examples of $Spin(7)$ holonomy metrics.\\

\noindent\textbf{Concluding Remarks.}
In this paper we have investigated a $\mathbb{T}^2$-reduction of torsion free $Spin(7)$ structures under the assumption that the quotient is K\"ahler. However as shown in section \ref{generalreduction} the quotient $SU(3)$-structure is generally only almost K\"ahler. Thus, it would interesting to investigate if other distinct types of $SU(3)$-structures, aside from the K\"ahler case considered here, can arise as well. From the results of this paper, it follows that, even locally, such a quotient cannot be a Calabi-Yau $3$-fold unless $N^8$ is the Riemannian product $P^6\times \mathbb{T}^2$. Furthermore, we have been able to prove that the quotient cannot be a special generalised CY $3$-fold as well. This still leaves plenty other cases to study. By contrast, in the $G_2$ case it is not hard to see that only two types of $SU(3)$-structures can arise, namely the K\"ahler one or the generic one i.e. with neither $\pi_1$ nor $\pi_2$ zero, in the notation of section \ref{generalreduction}. The latter case occurs, for instance, for the circle reduction of the Bryant-Salamon metric on the spinor bundle of $S^3$. Another interesting problem would be to investigate if one can find smooth completions to our $Spin(7)$ metrics. This will likely necessitate the study of non-free $\mathbb{T}^2$ actions.

\bibliography{biblioG}
\bibliographystyle{plain}

\end{document}